\documentclass[USenglish]{article}	

\usepackage[utf8]{inputenc}				
\usepackage[big,online]{dgruyter}	
\usepackage{lmodern} 
\usepackage{microtype}
\usepackage[numbers,square,sort&compress]{natbib}

\usepackage{enumerate}

\theoremstyle{dgthm}
\newtheorem{theorem}{Theorem}
\newtheorem{corollary}{Corollary}
\newtheorem{proposition}{Proposition}

\theoremstyle{dgdef}
\newtheorem{definition}{Definition}
\newtheorem{example}{Example}
\newtheorem{remark}{Remark}

\usepackage{color}
\usepackage{MnSymbol}

\newcommand{\rmd}{\mathrm{d}}

\begin{document}

	\articletype{Research Article}
	\received{\today}
	\revised{Month	DD, YYYY}
  \accepted{Month	DD, YYYY}
  \journalname{De~Gruyter~Journal}
  \journalyear{YYYY}
  \journalvolume{XX}
  \journalissue{X}
  \startpage{1}
  \aop
  \DOI{10.1515/sample-YYYY-XXXX}

\title{{Catenaries} and singular minimal surfaces in the simply isotropic space}

\author[1]{Luiz C. B. da Silva}
\author[2]{Rafael L\'opez}
\runningauthor{L. C. B. da Silva and R. L\'opez}
\affil[1]{\protect\raggedright 
Weizmann Institute of Science, Department of Physics of Complex Systems, Rehovot 7610001, Israel, e-mail: luiz.da-silva@weizmann.ac.il}
\affil[2]{\protect\raggedright 
Universidad de Granada, Departamento de Geometr\'ia y Topolog\'ia, 18071 Granada, Spain, e-mail: rcamino@ugr.es}
	

\abstract{
This paper investigates the hanging chain problem in the simply isotropic plane as well as its 2-dimensional analog in the simply isotropic space. The simply isotropic plane and space are two- and three-dimensional geometries equipped with a degenerate metric whose kernel has dimension 1. Although the metric is degenerate, the hanging chain and hanging surface problems are well-posed if we employ the relative arc length and relative area to measure the {weight}. Here, the concepts of relative arc length and relative area emerge by seeing the simply isotropic geometry as a relative geometry. In addition to characterizing the simply isotropic catenary, i.e., the solutions of the hanging chain problem, we also prove that it is the generating curve of a minimal surface of revolution in the simply isotropic space. Finally, we obtain the 2-dimensional analog of the catenary, the so-called singular minimal surfaces, and determine the shape of a hanging surface of revolution in the simply isotropic space.
}


\maketitle

\section{Introduction} 
\label{intro}

The problem of finding the shape of a hanging inextensible chain has received much attention from scientists since the times of Galileo. The problem was solved at the end of the XVII century by Hooke, Leibniz, Huygens, and Bernoulli, among others. The solution is the curve known as the catenary. Later, Euler proved in 1744 that the catenoid is the only non-planar minimal surface of revolution and that the generating curve of this surface is just the catenary. Thus, the catenary appears in two different scenarios, first as a solution to the physical problem of a hanging chain and second as a solution to the problem of finding the surfaces of revolution with minimum surface area. {Recently, the study of catenaries has been also extended to ambient spaces of constant curvature \cite{Lopez2022CatenarySpaceForms,Lopez2022CatenaryLorentzSpaceForms}, where, among other things, it is shown that non-zero curvature implies that solving the hanging chain problem is no longer equivalent to finding minimal surfaces of revolution.}  

In this paper, we propose the hanging chain problem in the simply isotropic plane $\mathbb{I}^2$. Similarly to the Euclidean setting, we will investigate whether the solution curve of this problem is the generating curve of a minimal surface of revolution of the $3$-dimensional simply isotropic space $\mathbb{I}^3$. For this last question, the plane $\mathbb{I}^2$ will be immersed in $\mathbb{I}^3$ and the reference line that serves to define the weight for curves of $\mathbb{I}^2$ will be the rotation axis of surfaces of revolution of $\mathbb{I}^3$.

The simply isotropic plane  $\mathbb{I}^2$ is the plane $\mathbb{R}^2$ endowed with the degenerate metric $\rmd s^2=\rmd x^2$ where $x$ and $z$ are the canonical Cartesian coordinates of $\mathbb{R}^2$. Analogously, the simply isotropic space $\mathbb{I}^3=\{(x,y,z)\in\mathbb{R}^3\}$ is the space equipped with the degenerate metric $\rmd s^2=\rmd x^2+\rmd y^2$.  The fact that the metric is degenerate is an obstacle regarding the minimization of the {weight} of a curve in $\mathbb{I}^2$. Indeed, let  $\gamma(x)=(x,z(x))$, $x\in [a,b]$, be a curve in $\mathbb{I}^2$. If we were to follow the same steps as in the Euclidean version of the hanging chain problem, first we would fix a reference line to measure the weight of a curve. Supposing that the reference line is the  $z$-axis, $L_z$, which is an isotropic line of $\mathbb{I}^2$, i.e., a line of length zero, the weight of a curve would be calculated by means of the distance between the points of $\gamma$ and $L_z$. The hanging chain would then minimize the {weight}, which correspond to the functional
\begin{equation*}
\mathcal{F}[\gamma]= \int_\gamma \mbox{dist}(\gamma(x),L_z)\, \rmd s-\lambda \int_\gamma\, \rmd s,
\end{equation*}
where $\mathcal{F}$ is defined among all curves with the same {length and} endpoints $\gamma(a)$ and $\gamma(b)$. The constant $\lambda$ is a Lagrange multiplier due to the length constraint. However, since $\mbox{dist}(\gamma(x),L_z)$ in $\mathbb{I}^2$ is just $\vert x\vert$, it follows that $\rmd s=\rmd x$, from which we conclude that  the functional $\mathcal{F}$ can be calculated explicitly, obtaining $\mathcal{F}=\frac{1}{2}(b^2-a^2)-\lambda(b-a)$. In other words, the functional {would be} constant for all curves with the same endpoints and the hanging chain problem would turn out to be trivial. A similar situation occurs if the reference line is the $x$-axis.

We can obtain a well-posed and non-trivial hanging chain problem by resorting to concepts of relative geometry. In $\mathbb{I}^2$, the points at a constant distance from a center form a pair of lines parallel to the $z$-axis, which does not lead to a manageable notion of a unit normal. The idea is to replace the unit circle with a curve of constant curvature 1, which in $\mathbb{I}^2$ is a parabola whose axis is parallel to the $z$-axis. By viewing $\mathbb{I}^2$ as a relative geometry, the role of the unit normal is then played by the vector connecting the focus of a unit parabola $\Sigma^1$ to the point of $\Sigma^1$ whose tangent line is parallel to the tangent of the curve, called the relative normal. As a byproduct of using a relative normal, one may introduce a new notion of arc length known as the relative arc length and denoted by $\rmd s^*$. In this work, we propose to replace the simply isotropic arc length element $\rmd s$ with the relative arc length element $\rmd s^*$ in the definition of the {weight} of a curve in $\mathbb{I}^2$. We will show that the modified {weight} functional leads to non-trivial solutions, where the graph of the natural logarithm then plays the role of a catenary in the simply isotropic space. Similarly to the Euclidean setting, the revolution of the graph of the logarithm around an isotropic axis also leads to a minimal surface in the simply isotropic space.


We may extend the hanging chain problem to surfaces whose solutions are called isotropic singular minimal surfaces. The surface is suspended under its weight measured with respect to a reference plane that can be isotropic or non-isotropic. As in the one-dimensional case, we replace the area element with the relative area. In contrast to the study of isotropic catenaries, we could not obtain an explicit parametrization of the isotropic singular minimal surfaces because the elliptic equation describing the surface cannot be integrated. However, we will focus on the classification of the isotropic singular minimal surfaces of revolution, obtaining a complete classification.


We divide this paper as follows. In Section  \ref{s-pre}, we revise the basic notions of the differential geometry of curves and surfaces in the simply isotropic spaces $\mathbb{I}^2$ and $\mathbb{I}^3$, with a particular emphasis on the notions of normal vector fields of curves and surfaces. In Section \ref{s-catenary}, we solve the hanging chain problem in the isotropic plane and obtain the concept of the isotropic catenary. In addition, we provide several characterizations of the isotropic catenary in terms of the curvature of the curve and certain vector fields of $\mathbb{I}^2$. In Section \ref{s-catenary2}, by rotating isotropic catenaries around an isotropic axis, we obtain minimal surfaces of revolution in $\mathbb{I}^3$. In Section \ref{s-surface}, we extend the hanging chain problem to dimension two, where the concept of isotropic singular minimal surface arises as a generalization of the isotropic catenary. As in the one-dimensional problem, it is necessary to distinguish between the cases where the reference plane is isotropic and non-isotropic.  Finally, Section \ref{s-invariant}  is devoted to fully classifying invariant singular minimal surfaces. First, we show that there exist no helicoidal singular minimal surfaces. Then, for surfaces of revolution in $\mathbb{I}^3$, we characterize Euclidean surfaces of revolution in Subsection \ref{s-elliptic} and {surfaces of} parabolic revolution in Subsection \ref{s-parabolic}.
 
 \section{Preliminaries}
\label{s-pre}

In this section, we review some basic notions of the differential geometry of curves and surfaces in the simply isotropic spaces $\mathbb{I}^2$ and $\mathbb{I}^3$. For more details, we refer the reader to \cite{si1,daSilvaMJOU2021}, or to \cite{Sachs1987,Sachs1990} for textbook sources (in German), though part of the discussion on curves appears here for the first time, to the best of our knowledge. 

The \emph{simply isotropic space} $\mathbb{I}^3$ corresponds to the canonical real vector space $\mathbb{R}^3$ with Cartesian coordinates $(x,y,z)$, and   equipped with the degenerate metric
\begin{equation}\label{def::IsotMetric}
    \langle \mathbf{u},\mathbf{v}\rangle = u^1v^1+u^2v^2,
\end{equation}
where $\mathbf{u}=(u^1,u^2,u^3)$ and $\mathbf{v}=(v^1,v^2,v^3)$. A non-zero vector $\mathbf{v}$ is said to be \emph{isotropic} if $\langle\mathbf{v},\mathbf{v}\rangle=0$. In addition, on the set of isotropic vectors, i.e., $\{\mathbf{v}\in\mathbb{I}^3:\mathbf{v}=(0,0,u^3)\}$, we shall use the secondary metric
\begin{equation}\label{def::2ndIsotMetric}
     \llangle \mathbf{u},\mathbf{v}\rrangle = u^3v^3.
\end{equation}
Therefore, the space $\mathbb{I}^3$ is an example of a Cayley-Klein vector space \cite{StruveRM2005}. 

A plane is said to be \emph{isotropic} if it contains an isotropic vector. Thus, in $\mathbb{I}^3$ an isotropic vector and an isotropic plane are vertical, that is, parallel to the $z$-axis. The inner product induces a semi-norm   $\Vert u\Vert=\sqrt{\langle u,u\rangle}$. The \emph{top-view projection} of a vector $\mathbf{u}$ is the projection $\tilde{\mathbf{u}}$ over the $xy$-plane,  
\begin{equation}\label{def:TopViewProjection}
\mathbf{u}=(u^1,u^2,u^3)\longmapsto \tilde{\mathbf{u}}\equiv(u^1,u^2,0). \end{equation}
In the following, it will be useful to resort to the Euclidean inner and vector products respectively written as 
\begin{equation*}
    \mathbf{u}\cdot\mathbf{v} = u^1v^1+u^2v^2+u^3v^3
\quad\mbox{ and }\quad    \mathbf{u}\times\mathbf{v} = (u^2v^3-u^3v^2,u^3v^1-u^1v^3,u^1v^2-u^2v^1).
\end{equation*}

Analogously,   the \emph{simply isotropic plane} $\mathbb{I}^2$ is defined as the canonical real vector space $\mathbb{R}^2=\{(x,z):x,z\in\mathbb{R}\}$ equipped with the degenerate metric
\begin{equation*}
    \langle \mathbf{u},\mathbf{v}\rangle = u^1v^1,
\end{equation*}
where $\mathbf{u}=(u^1,u^3)$ and $\mathbf{v}=(v^1,v^3)$. The secondary metric $\mathbb{I}^2$ is also defined as in Eq. \eqref{def::2ndIsotMetric}. Note we can alternatively see $\mathbb{I}^2$ isometrically embedded in $\mathbb{I}^3$ as the surface implicitly defined by the equation $y=0$.

\subsection{Geometry of curves in the simply isotropic plane}

Let $I\subset\mathbb{R}$ denote an interval and $\gamma(t)=(x(t),z(t))$, $t\in I$, a smooth curve in $\mathbb{I}^2$. In what follows, we shall restrict our discussion to \emph{admissible curves}, i.e., curves whose tangent lines are not isotropic, i.e., $\rmd x/\rmd t\not=0$ for all $t\in I$. Note that a curve $\gamma:I\to\mathbb{I}^2$ is parametrized by arc length $s$ if 
\begin{equation*}
     \gamma(s)=(\pm s,z(s)). 
\end{equation*}
The normal to $\gamma$ with respect to the simply isotropic metric is the isotropic vector $\mathcal{N}=(0,1)$, which is normalized by the secondary metric: $\llangle\mathcal{N},\mathcal{N}\rrangle=1$. The unit tangent is $\mathbf{T}(s)=(\pm1,z'(s))$. Thus  
\begin{equation*}
    \mathbf{T}' = (0,z'')=\kappa\,\mathcal{N},\quad\kappa=z''.
\end{equation*}
The function $\kappa$ is the (signed) \emph{simply isotropic curvature} of $\gamma$.

For a generic parameter $t$, let us write $\gamma(t)=(x(t),z(t))$. We shall distinguish between derivatives with respect to the arc length parameter $s$ and a generic parameter $t$ by respectively using a prime and a dot: e.g.,  $x'=\rmd x/\rmd s$ and $\dot{x}=\rmd x/\rmd t$. For a generic parametrization, the unit tangent becomes
\begin{equation*}
    \mathbf{T}=\frac{\dot{\gamma}}{\Vert\dot{\gamma}\Vert}=\frac{(\dot{x},\dot{z})}{\sqrt{\dot{x}^2}} =   \left (\pm 1,\frac{\dot{z}}{\dot{x}}\right).  
\end{equation*}
On the other hand, to find the curvature, we use the chain rule
\begin{equation}\label{eq::3dNmin}
    \kappa\,\mathcal{N} = \mathbf{T}' = \frac{\rmd t}{\rmd s}\,\dot{\mathbf{T}}=\frac{1}{\dot{x}}\left(0,\frac{\ddot{z}\dot{x}-\dot{z}\ddot{x}}{\dot{x}^2}\right).
\end{equation}
Therefore, the curvature function of $\gamma(t)=(x(t),z(t))$, $\dot{x}>0$, is given by 
\begin{equation}\label{curvature}
\kappa = \frac{\dot{x}\ddot{z}-\ddot{x}\dot{z}}{\dot{x}^3}.
\end{equation}
From now on, we may assume for simplicity that $\dot{x}>0$, after applying a simply isotropic rigid motion $(x,z)\mapsto(\pm x+a,z)$ if necessary.

\begin{example}
The parabola $\gamma(t)=(t,\frac{1}{2}ct^2+bt+a)$ has constant curvature $c$.   Conversely, any curve with constant curvature is a parabola. Indeed, if $\kappa=c$, and if we write $\gamma(t)=(t,z(t))$, then $\kappa=c=z''(t)$, so $z(t)=\frac{1}{2}ct^2/2+bs+a$, $a,b,c\in\mathbb{R}$. Here, we include the case $c=0$, which corresponds to $\gamma$ being a straight-line. \qed
\end{example}

For a curve $\gamma:I\to\mathbb{E}^2$ in the Euclidean plane $\mathbb{E}^2$ parametrized by arc length $s$, the (signed) curvature is computed as $\kappa=\gamma''\cdot\mathbf{N}$, where $\mathbf{N}$ is a unit vector field normal to the curve. For a generic parameter $t$, we define the first and second fundamental forms of $\gamma$ as $g_{11}=\dot{\gamma}\cdot\dot{\gamma}$ and $h_{11}=\ddot{\gamma}\cdot\mathbf{N}$, respectively. Then, using the chain rule
\begin{equation}
    \kappa = \gamma''\cdot\mathbf{N} = \left[\left(\frac{\rmd t}{\rmd s}\right)^2\ddot{\gamma}+\frac{\rmd t}{\rmd s}\frac{\rmd}{\rmd t}\left(\frac{\rmd t}{\rmd s}\right)\dot{\gamma}\right]\cdot\mathbf{N} = \left(\frac{\rmd t}{\rmd s}\right)^2\ddot{\gamma}\cdot\mathbf{N} = \frac{h_{11}}{g_{11}}.
\end{equation}

Analogously, we define the first fundamental form for curves $\gamma(t)=(x(t),z(t))$ in $\mathbb{I}^2$ as $g_{11}=\langle\dot{\gamma},\dot{\gamma}\rangle=\dot{x}^2$. To introduce the second fundamental form, we may enforce the condition that $\kappa=h_{11}/g_{11}$:
\begin{equation*}
     \kappa = \frac{\ddot{z}\dot{x}-\dot{z}\ddot{x}}{\dot{x}^3} = \frac{1}{\dot{x}^2}\frac{\ddot{z}\dot{x}-\dot{z}\ddot{x}}{\dot{x}} \Rightarrow h_{11} = \frac{\ddot{z}\dot{x}-\dot{z}\ddot{x}}{\dot{x}}.
\end{equation*}
Note that we can not obtain the coefficient $h_{11}$ by taking an inner product  of the acceleration vector $\ddot{\gamma}$ with some normal vector. To achieve that, we may resort to the metric of the Euclidean plane. Indeed, if we define the \emph{minimal normal} vector field (see Figure \ref{fig:NminNpar})
\begin{equation}\label{min}
\mathbf{N}_{min} = \left(-\frac{\dot{z}}{\dot{x}},1\right) = \frac{1}{\dot{x}}J(\dot{\gamma}),
\end{equation}
where $J$ is the counter-clockwise (Euclidean) $\frac{\pi}{2}$-rotation on the $xz$-plane, we can finally write the second fundamental form of a curve $\gamma:I\to\mathbb{I}^2$ as
\begin{equation}
    h_{11}=\ddot{\gamma}\cdot\mathbf{N}_{min}=\det(\dot{\gamma},\ddot{\gamma}).
\end{equation}

\begin{remark}
The terminology ``minimal normal'' comes from the similar construction for surfaces in the space $\mathbb{I}^3$ \cite{KelleciJMAA2021}. More precisely, if one tries to see $-\rmd\mathbf{N}_{min}$ as a shape operator, its trace vanishes identically for \emph{any} surface in $\mathbb{I}^3$. (See the next subsection for further details on the simply isotropic geometry of surfaces.)
\end{remark}

\subsection{Simply isotropic relative geometry}

To further study the geometry of curves and surfaces in the simply isotropic space, we shall resort to some ideas from the so-called relative geometry \cite{Simon1991}, a topic whose origins can be traced back to the contributions of E. M\"uller in the 1920s \cite{MullerMonatsh}. The reader may consult \cite{Simon1991} or Sect. 2 of \cite{VerpoortAG2012} for further information.

The basic idea of relative geometry is that in Euclidean space many properties associated with the normal $\mathbf{N}$ to a hypersurface $\mathbf{x}:U\subset\mathbb{R}^n\to M^n\subset\mathbb{E}^{n+1}$ do not depend on the orthogonality. For example, the definition of the Christoffel symbols of an affine connection relies on the property that $\mathbf{N}$ is transversal to the tangent space: $\mathbf{x}_{ij}=\Gamma_{ij}^k\mathbf{x}_k+h_{ij}\mathbf{N}$. As another example, the definition of the shape operator, and consequently of the Gaussian and mean curvatures, relies on the property that $T_pM$ and $T_{\mathbf{x}(p)}\mathbb{S}^n$ are parallel and that $-\rmd\mathbf{N}$ takes values on $T_pM^2$. 

\begin{definition}
 Let $M^n\subset\mathbb{R}^{n+1}$ be a manifold and $\mathbf{y}$ a vector field along $M$  taking values in $\mathbb{R}^{n+1}$. The vector field $\mathbf{y}$ is a \emph{relative normal} of $M$ if it is (i) transversal to $T_pM$, i.e., $\mathbf{y}\not\in T_pM$ and (ii) equiaffine, i.e., $\forall\,v\in T_pM,\,\rmd\mathbf{y}(v)\in T_pM$.
\end{definition}

Let $\bar{\Sigma}^n\subset\mathbb{R}^{n+1}$ be a hypersurface with relative normal $\bar{\mathbf{y}}$.  The \emph{Peterson mapping}  of a hypersurface $M^n$ is a map $\mathcal{P}$ that sends a point $p\in M$ to a point $q\in\bar{\Sigma}^n$ such that $T_pM$ and $T_q\bar{\Sigma}$ are parallel. We may introduce a relative normal $\mathbf{y}$ for $M$ with respect to $\bar{\Sigma}$ by defining
\begin{equation*}
    \mathbf{y}(p) = \bar{\mathbf{y}}(\mathcal{P}(p)).
\end{equation*}
The \emph{relative shape operator} $A$ of $M$ is then given by $A=-\rmd\mathbf{y}$. In this work, we shall concentrate on the use of the centro-affine normal, i.e., $\bar{\mathbf{y}}$ is the position vector of $\bar{\Sigma}$. We shall refer to $\bar{\Sigma}$ as the \emph{relative sphere}. 

In the relative approach to the simply isotropic plane, we may take as the relative sphere the circle of parabolic type
\begin{equation}\label{def::UnitCircleParabolicType}
    \Sigma^1 = \left\{(x,z)\in\mathbb{I}^2:z=\frac{1}{2}-\frac{x^2}{2}\right\}.
\end{equation}
We then obtain a relative normal $\mathbf{N}_{par}$, named the {parabolic normal}, as an alternative to the minimal normal. More precisely, the \emph{parabolic normal} vector field $\mathbf{N}_{par}$ is defined as (see Figure \ref{fig:NminNpar})
\begin{equation}\label{par}
    \mathbf{N}_{par} = \left(-\frac{\dot{z}}{\dot{x}},\frac{1}{2}-\frac{\dot{z}^2}{2\dot{x}^2}\right).
\end{equation}

\begin{figure}
    \centering
    \includegraphics[width=0.4\linewidth]{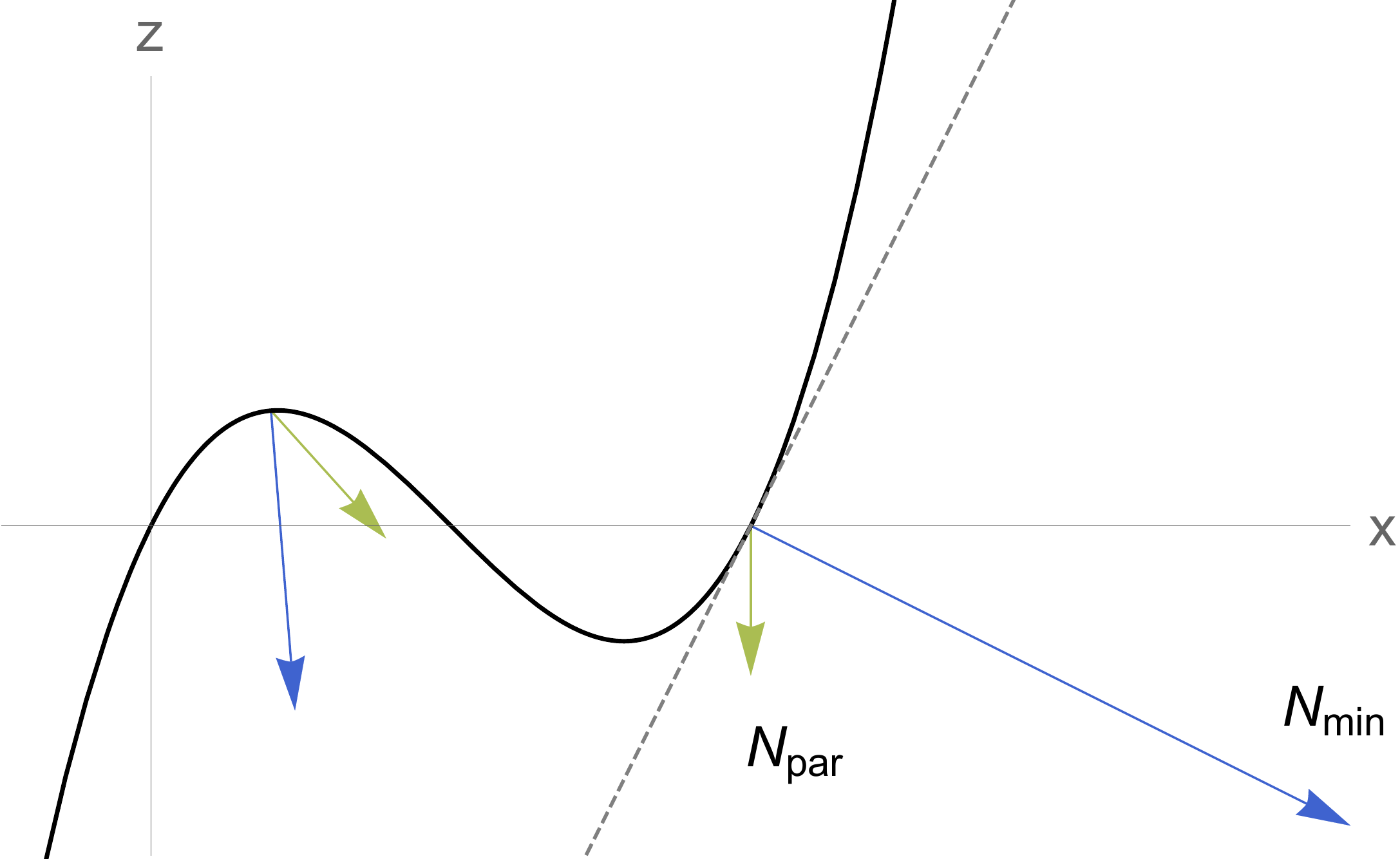}
    \caption{Minimal ($\mathbf{N}_{min}$) and parabolic ($\mathbf{N}_{par}$) normal vector fields of an admissible curve in $\mathbb{I}^2$. Note that $\mathbf{N}_{min}$, Eq. \eqref{min}, is a multiple of the Euclidean normal, while $\mathbf{N}_{par}$, Eq. \eqref{par}, does not come from a notion of orthogonality. Figure generated with Mathematica.}
    \label{fig:NminNpar}
\end{figure}

\begin{proposition}
The parabolic normal $\mathbf{N}_{par}$ of an admissible curve $\gamma:I\to\mathbb{I}^2$ is a relative normal.
\end{proposition}
\begin{proof}
We need to prove that $\mathbf{N}_{par}$ is equiaffine and transversal. The parabolic normal $\mathbf{N}_{par}$ is equiaffine:
\begin{equation*} 
-\frac{\rmd\mathbf{N}_{par}}{\rmd t} = \left(\frac{\ddot{z}\dot{x}-\dot{z}\ddot{x}}{\dot{x}^2},\frac{\dot{z}}{\dot{x}}\frac{\ddot{z}\dot{x}-\dot{z}\ddot{x}}{\dot{x}^2}\right) = \kappa\,\dot{\gamma}.
\end{equation*}
Finally, the parabolic normal $\mathbf{N}_{par}$ is transversal because
\[
\det\left(\begin{array}{cc}
    \dot{x} & \dot{z} \\
    -\frac{\dot{z}}{\dot{x}} & \frac{1}{2}-\frac{\dot{z}^2}{2\dot{x}^2}\\
\end{array}
\right) = \frac{\dot{x}^2+\dot{z}^2}{2\dot{x}}\not=0,
\]
where we used that $\dot{x}\not=0$.
\end{proof}

As a corollary from the fact that $\mathbf{N}_{par}$ is a relative normal, it follows we may alternatively compute the second fundamental form of $\gamma$ as
\begin{equation}
    h_{11} = \kappa \langle\dot{\gamma},\dot{\gamma}\rangle = \langle -\rmd\mathbf{N}_{par}(\dot{\gamma}),\dot{\gamma}\rangle.
\end{equation}

An important concept in the context of relative geometry of surfaces is played by the \emph{relative area} $A^*$. More precisely, given a surface $M^2:(u,v)\mapsto \mathbf{r}(u,v)$ with relative normal $\mathbf{y}$, we define
\begin{equation}\label{def::RelativeSurfArea}
    A^* = \int_M \vert\det(\mathbf{r}_u,\mathbf{r}_v,\mathbf{y})\vert\,\rmd u\rmd v.
\end{equation}
Analogously, we define the \emph{relative arc length} $s^*$ for a curve $\gamma:[a,b]\to\mathbb{I}^2$ with parabolic normal $\mathbf{N}_{par}$ as
\begin{equation}
    s^* = \int_a^b \vert \det(\mathbf{N}_{par},\dot{\gamma}) \vert\,\rmd t = \int_a^b \vert \mathbf{N}_{par}\cdot J(\dot{\gamma}) \vert\,\rmd t = \int_a^b \vert \mathbf{N}_{par}\cdot \mathbf{N}_{min} \vert\,\dot{x}\,\rmd t.
\end{equation}
Thus, from Eqs. \eqref{min} and \eqref{par}, the relative arc length element of a curve $\gamma(t)=(x(t),z(t))$ satisfies
\begin{equation}\label{ds}
\rmd s^*=(\mathbf{N}_{par}\cdot \mathbf{N}_{min})\, \rmd s=\left(\frac{1}{2}+\frac{\dot{z}^2}{2\dot{x}^2}\right)\, \rmd s=\left(\frac{\dot{x}}{2}+\frac{\dot{z}^2}{2\dot{x}}\right)\, \rmd t.
\end{equation}
\begin{remark}
In Euclidean plane, the relative arc length with respect to the relative normalization given by the Euclidean normal coincides with the usual arc length: $s^*=\int\det(\mathbf{N},\dot{\gamma})\,\rmd t=\int \mathbf{N}\cdot\mathbf{N}\sqrt{\dot{x}^2+\dot{z}^2}\,\rmd t = \int \rmd s=s$, where we used that $\mathbf{N}=J(\dot{\gamma})/\sqrt{\dot{\gamma}\cdot\dot{\gamma}}$ and $J(\mathbf{u})\cdot \mathbf{v}=\det(\mathbf{u},\mathbf{v})$.
\end{remark}

\subsection{Geometry of surfaces in the simply isotropic space}  
 
Let $U\subset\mathbb{R}^2$ be an open set and $\mathbf{r}:U\to M^2\subset\mathbb{I}^3$ a regular parametrized surface. In what follows, we shall focus on \emph{admissible surfaces}, that is, surfaces whose tangent planes are not isotropic. Equivalently,  $x_1^1x_2^2-x_2^1x_1^2\not=0$, where $\mathbf{r}(u^1,u^2)=(x^1(u^1,u^2),x^2(u^1,u^2),x^3(u^1,u^2))$ and $x_i^j=\partial x^j/\partial u^i$. Note that every admissible surface in $\mathbb{I}^3$ is locally parametrized as the graph of a real function $\mathbf{r}=(u^1,u^2,f(u^1,u^2))$, called the \emph{normal form} of $M^2$.

The \emph{first fundamental form} is defined as usual $g_{ij}=\langle\mathbf{r}_i,\mathbf{r}_j\rangle$, where $\mathbf{r}_i=\partial\mathbf{r}/\partial u^i$. The normal to a surface $M^2$ with respect to the simply isotropic metric is given by the isotropic vector $\mathcal{N}=(0,0,1)$ normalized by the secondary metric. The \emph{second fundamental form} $h_{ij}$ is defined by the expression
\begin{equation}
    \mathbf{r}_{ij} = \Gamma_{ij}^k\mathbf{r}_k+h_{ij}\mathcal{N},
\end{equation}
where we shall adopt the convention of summing over repeated indices. If we parametrize $M$ in its normal form, the first and second fundamental forms are given by 
\begin{equation}
     \mathrm{I}= (\rmd u^1)^2+(\rmd u^2)^2\quad\mbox{and}\quad  \mathrm{II} = f_{ij}\rmd u^i\rmd u^j.
\end{equation}

We can compute the second fundamental form $h_{ij}$ by forcing an analogy with Euclidean space:
\begin{equation*}
    h_{ij} = \frac{\partial^2\mathbf{r}}{\partial u^i\partial u^j}\cdot \mathbf{N}_{min} =\frac{\det(\mathbf{r}_1,\mathbf{r}_2,\mathbf{r}_{ij})}{\sqrt{\det g_{ij}}},
\end{equation*}
where $\mathbf{N}_{min}$ is the \emph{minimal normal}
\begin{equation}
    \mathbf{N}_{min} = \frac{\mathbf{r}_1\times\mathbf{r}_2}{\sqrt{g_{11}g_{22}-g_{12}^2}}.
\end{equation}
The map $\mathbf{N}_{min}$ is not a relative normal. Indeed, the minimal normal $\mathbf{N}_{min}$ is not equiaffine since $\rmd\mathbf{N}_{nim}$ is always a horizontal vector and, as a consequence, it generally fails to be tangent to $M^2$: 
\[
\mathbf{N}_{min}=\left(\frac{X_{23}}{X_{12}},\frac{X_{31}}{X_{12}},1\right),\quad X_{ij}=\det\left(
       \begin{array}{cc}
       x_1^i & x_1^j \\
       x_2^i & x_2^j \\
       \end{array}
       \right),
\]
where we may assume $X_{12}>0$ by exchanging $u^1\leftrightarrow u^2$ if necessary. The \emph{mean curvature} $H$ is given by 
\begin{equation}\label{mean}
    H=\frac12\frac{g_{11}h_{22}-2g_{12}h_{12}+g_{22}h_{11}}{g_{11}g_{22}-g_{12}^2}.
\end{equation}

In the simply isotropic space, we may take as the relative sphere the sphere of parabolic type
\begin{equation}\label{def::UnitSphereParabolicType}
    \Sigma^2 = \left\{(x,y,z)\in\mathbb{I}^3:z=\frac{1}{2}-\frac{x^2}{2}-\frac{y^2}{2}\right\}.
\end{equation}
We then obtain a relative normal $\mathbf{N}_{par}$, named \emph{parabolic normal}, as an alternative to the minimal normal. The \emph{parabolic normal} vector field $\mathbf{N}_{par}$ is defined as
\begin{equation}\label{par2}
    \mathbf{N}_{par} = \left(\frac{X_{23}}{X_{12}},\frac{X_{31}}{X_{12}},\frac{1}{2}-\frac{X_{23}^2+X_{31}^2}{2X_{12}^2}\right).
\end{equation}

\begin{proposition}
The parabolic normal $\mathbf{N}_{par}$ of an admissible surface $\mathbf{r}:U\to M^2\subset\mathbb{I}^3$ is a relative normal.
\end{proposition}
\begin{proof}
From the definition of $\mathbf{N}_{par}$ in \eqref{par2}, we have
\[
\det(\mathbf{r}_1,\mathbf{r}_2,\mathbf{N}_{par}) = \frac{(X_{23})^2+(X_{31})^2+(X_{12})^2}{2X_{12}}>0,
\]
which implies  $\mathbf{N}_{par}$ is transversal to $M^2$. Let us parametrize $M^2$ in its normal form $\mathbf{r}=(u^1,u^2,f(u^1,u^2))$. Then, the parabolic normal is given by $\mathbf{N}_{par}=(-f_1,-f_2,\frac{1}{2}-\frac{f_1^2+f_2^2}{2})$, from which follows that
\[
-\frac{\partial\mathbf{N}_{par}}{\partial u^i} = f_{1i}\,\mathbf{r}_1+f_{2i}\,\mathbf{r}_2. 
\]
Therefore, $\mathbf{N}_{par}$ is equiaffine.
\end{proof}

As a byproduct of using the parabolic normal $\mathbf{N}_{par}$, we may alternatively compute the second fundamental form and mean curvature of $M^2$ as
\begin{equation}
    h_{ij} = \langle -\rmd\mathbf{N}_{par}(\mathbf{r}_j),\mathbf{r}_i\rangle \quad \mbox{and} \quad H = \mbox{tr}(-\rmd\mathbf{N}_{par}).
\end{equation}
In addition, the \emph{Gaussian curvature} of  $M^2$ is given by $K=\det(-\rmd\mathbf{N}_{par})=\frac{\det h_{ij}}{\det g_{ij}}$.

Now, let us compute the relative area of $M^2$. From the definition \eqref{def::RelativeSurfArea}, we have
\begin{equation}\label{def::RelAreaInI3}
    A^* = \int_M \det(\mathbf{r}_1,\mathbf{r}_2,\mathbf{N}_{par})\,\rmd u^1\rmd u^2 = \int_M \mathbf{N}_{min}\cdot\mathbf{N}_{par}\,\sqrt{g_{11}g_{22}-g_{12}^2}\rmd u^1\rmd u^2 = \int_M (\mathbf{N}_{min}\cdot\mathbf{N}_{par})\rmd A.
\end{equation}
If we parametrize $M$ in its normal form, the relative area takes the form
\begin{equation}
    A^* = \int\frac{1+f_1^2+f_2^2}{2}\,\rmd u^1\rmd u^2.
\end{equation}

\section{The solution of the simply isotropic hanging chain problem}\label{s-catenary}
\label{subsect::CatenaryAsHangingChain}

In the hanging chain problem in $\mathbb{I}^2$, the {weight} of the curve is calculated using the distance to a straight-line $L$ of  $\mathbb{I}^2$. Since there are isotropic and non-isotropic  straight-lines, it will be necessary to distinguish between both cases.  Without loss of generality, if $L$  is isotropic, it will be assumed  to be the $z$-axis $L_z=\{(0,z):z\in\mathbb{R}\}$ and if $L$ is non-isotropic, then the line will be the $x$-axis $L_x=\{(x,0):x\in\mathbb{R}\}$. In each case, the functional 
\begin{equation}
    \mathcal{F}[\gamma]= \int_\gamma \mbox{dist}(\gamma(x),L)\, \rmd s^*-\lambda \int_\gamma\, \rmd s^*
\end{equation}
will be denoted by $\mathcal{F}_z$ and $\mathcal{F}_x$. In addition, all curves will be contained in one of the two half-planes determined by $L_z$  and $L_x$, namely, $ \{(x,z)\in\mathbb{R}^2:x>0\}$ and $ \{(x,z)\in\mathbb{R}^2:z>0\}$, respectively. 

The first case to consider is when the reference line is the isotropic line $L_z$.  
\begin{theorem}\label{t1} 
The critical points of the functional $\mathcal{F}_z$ are the curves 
\begin{equation}\label{cate}
   \gamma(t) = (t,c\ln( t-\lambda)+d),
\end{equation}
where $c,d\in\mathbb{R}$.
\end{theorem}

From now on, the curves \eqref{cate} when $c\not=0$ will be called \emph{isotropic catenaries} with respect to $L_z$.

 \begin{proof}

Let us consider the parametrization $\gamma(t)=(t,z(t))$. Using \eqref{ds}, the relative arc length element becomes
\begin{equation*}
    ds^* =  \left(\frac{{1}}{2}+\frac{\dot{z}^2}{2}\right)\rmd t.
\end{equation*}
The functional $\mathcal{F}_z$  writes as  
\begin{equation*}
    \mathcal{F}_z[\gamma] =  \int_a^b(t-\lambda)\left(\frac{{1}}{2}+\frac{\dot{z}^2}{2}\right)\rmd t.
\end{equation*}
The Euler-Lagrange equation is 
\begin{equation}\label{el}
\frac{\partial F}{\partial z}-\frac{\rmd}{\rmd t}\left(\frac{\partial F}{\partial \dot{z}}\right)=0,
\end{equation}
where 
$$ F(t,z,\dot{z})= (t-\lambda)\left(\frac{{1}}{2}+\frac{\dot{z}^2}{2}\right).$$
Notice that $F$ does not depend on $z$. Thus, the Euler-Lagrange equation leads to
\begin{equation}\label{cate3}
0=\frac{\rmd}{\rmd t}\left((t-\lambda)\dot{z}\right)=\dot{z}+(t-\lambda)\ddot{z}.
\end{equation}
Then, there exists $c\in\mathbb{R}$ such that $(t-\lambda)\dot{z}=c$. If $c=0$, then $z$ is a constant function, $z(t)=d$, $d\in\mathbb{R}$. This is a particular case of \eqref{cate}.  If $c\not=0$, then a direct integration gives Eq. \eqref{cate}.
\end{proof}

 Isotropic catenaries with respect to $L_z$ can be also characterized as solutions of a coordinate-free prescribed curvature problem involving the curvature $\kappa$ of $\gamma$, the parabolic normal vector ${\bf N}_{par}$, and the unit vector field of $\mathbb{I}^2$ which is orthogonal to $L_z$. Let $X=\partial_x\in\mathfrak{X}(\mathbb{I}^2)$.  

\begin{theorem}\label{t2} Let $\gamma$ be a curve in $\mathbb{I}^2$. Then, $\gamma$ is an isotropic catenary with respect to $L_z$ if, and only if, its curvature $\kappa$ satisfies
\begin{equation}\label{cate2}
\kappa=\frac{\langle\mathbf{N}_{par},X\rangle}{\langle\gamma,X\rangle-\lambda}.
\end{equation}
\end{theorem}

\begin{proof} From Eq.  \eqref{curvature}, the curvature of $\gamma(t)=(t,z(t))$ is $\kappa=\ddot{z}$. This implies that Eq. \eqref{cate3} writes simply as 
$$\kappa=-\frac{\dot{z}}{t-\lambda}.$$
Finally, from Eq. \eqref{min}, we have $\mathbf{N}_{par}=(-\dot{z},\frac{1}{2}-\frac{\dot{z}^2}{2})$, hence  $\langle\mathbf{N}_{par},X\rangle=- \dot{z}$.  These expressions prove the validity of Eq. \eqref{cate2}.
\end{proof}

Naturally, the solutions of Eq. \eqref{cate3} are given by Eq. \eqref{cate}. The notion of catenary can be generalized if it is introduced a power $\alpha\in\mathbb{R}$ in the functional $\mathcal{F}_z$. More precisely, define   the functional
\begin{equation*}
    \mathcal{F}_z^\alpha[\gamma] = \int_\gamma (z^\alpha-\lambda)\, \rmd s^*.
    \end{equation*}
  \begin{theorem} \label{t3}
 Let $\gamma$ be a curve in $\mathbb{I}^2$. Then, $\gamma$ is a critical point of $\mathcal{F}_z^\alpha$ if, and only if, its curvature $\kappa$ satisfies
\begin{equation}\label{cate4}
\kappa=\alpha\frac{\langle\gamma,X\rangle^{\alpha-1}\langle\mathbf{N}_{par},X\rangle}{\langle\gamma,X\rangle^\alpha-\lambda},
\end{equation}
 \end{theorem}
\begin{proof}
The computations are similar to those in the proof of Theorem \ref{t2}. If $\gamma(t)=(t,z(t))$, then the functional $\mathcal{F}_z^\alpha$ becomes 
\begin{equation*}
    \mathcal{F}_z^\alpha[\gamma] = \int_a^b  \left(t^{\alpha}-\lambda\right)\left(\frac{{1}}{2}+\frac{\dot{z}^2}{2}\right)\rmd t.
    \end{equation*}
The computation of the Euler-Lagrange equation \eqref{el} gives 
\begin{equation}\label{cate42}
\alpha t^{\alpha-1}\dot{z}+(t^\alpha-\lambda)\ddot{z}=0,
\end{equation} 
proving the theorem.
\end{proof}

By analogy with the Euclidean space \cite{Dierkes1990,LopezAGAG2018}, and taking $\lambda=0$ in Eq. \eqref{cate4}, we give the following definition: 
 
 \begin{definition} A curve $\gamma$ in $\mathbb{I}^2$ is said to be an \emph{isotropic $\alpha$-catenary} with respect to $L_z$ if its curvature satisfies 
 \[ \kappa=\alpha\frac{  \langle\mathbf{N}_{par},X\rangle}{\langle\gamma,X\rangle  }.\]
 \end{definition}

The case $\alpha=1$ corresponds to the isotropic catenary \eqref{cate}. In fact, Eq. \eqref{cate42} which characterizes $\alpha$-catenaries can be solved and the solution provides an explicit parametrization. Indeed,
  
\begin{corollary} 
Let $\gamma$ be an $\alpha$-catenary. Then, it is parametrized as
  \begin{equation}\label{acate1}
     \gamma(t) = \begin{cases}
         (t,c\ln t+d ), & \text{if} \quad \alpha=1  \\
         (t,c\,t^{1-\alpha} +d), & \text{if} \quad \alpha\not=0 
      \end{cases},
  \end{equation}
where $c,d\in\mathbb{R}$.
\end{corollary}

We conclude this section by investigating the hanging chain problem when the reference line is   the non-isotropic line $L_x$. In this case, we must resort to the distance as computed using the secondary metric: $d(\bf{x},\bf{y})=\sqrt{\llangle\bf{x}-\bf{y},\bf{x}-\bf{y}\rrangle}$. Thus, the functional to minimize is
\begin{equation}\label{fx}
    \mathcal{F}_x^{\alpha} = \int_{\gamma}(z^{\alpha}-\lambda)\,\rmd s^*.
\end{equation}
Now, the computations follow the same steps as in Theorems \ref{t1} and \ref{t2}. Let $Z=\partial_z\in\mathfrak{X}(\mathbb{I}^2)$.

\begin{theorem}\label{t12} Let $\gamma$ be a curve in $\mathbb{I}^2$ parametrized by $\gamma(t)=(t,z(t))$. The following statements are equivalent:
\begin{enumerate}
\item The curve $\gamma$ is a critical point of the functional $\mathcal{F}_x^{\alpha}$;
\item The function $z$ satisfies 
\begin{equation}\label{cate5}
(z^{\alpha}-\lambda)\ddot{z} = \alpha z^{\alpha-1}\left(\frac{1}{2}-\frac{\dot{z}^2}{2}\right),
\end{equation}
\item The curvature $\kappa$ satisfies
$$  \kappa=\alpha\llangle \gamma,Z\rrangle^{\alpha-1}\frac{ \llangle \mathbf{N}_{par},Z\rrangle}{\llangle \gamma,Z\rrangle^{\alpha}-\lambda}. $$

\end{enumerate}
\end{theorem}
These curves will be called  \emph{isotropic $\alpha$-catenaries} with respect to $L_x$.

 \begin{proof} 
The functional $\mathcal{F}_x^{\alpha}$ in \eqref{fx}  is
\begin{equation*}
    \mathcal{F}_x^{\alpha}[\gamma] = \int_{\gamma}(z^{\alpha}-\lambda)\,\rmd s^* = \int_a^b(z^{\alpha}-\lambda) \left(\frac{1}{2}+\frac{\dot{z}^2}{2}\right)\rmd t.
\end{equation*}
The corresponding Euler-Lagrange equation is 
\begin{equation}
\alpha z^{\alpha-1}\frac{1+\dot{z}^2}{2} = \frac{\rmd}{\rmd t}\Big[(z^{\alpha}-\lambda)\dot{z}\Big] = \alpha z^{\alpha-1}\dot{z}^2+(z^{\alpha}-\lambda)\ddot{z},
\end{equation}
which finally gives
\begin{equation}\label{cate66}
(z^{\alpha}-\lambda)\ddot{z}=\alpha z^{\alpha-1}\left(\frac{1}{2}-\frac{\dot{z}^2}{2}\right) \Rightarrow \ddot{z} = \alpha z^{\alpha-1}\frac{\frac{1}{2}-\frac{\dot{z}^2}{2}}{z^{\alpha}-\lambda}.
\end{equation}
Since $\llangle\mathbf{N}_{par},Z\rrangle=\llangle(-\dot{z},\frac{1}{2}-\frac{\dot{z}^2}{2}),(0,1)\rrangle=\frac{1}{2}(1-\dot{z}^2)$, $\llangle\gamma,Z\rrangle=z$, and the curvature is $\kappa=\ddot{z}$, we finally obtain the desired chain of equivalences.
\end{proof}

\begin{remark}
It is worth mentioning that there is an isometry between the simply isotropic plane $\mathbb{I}^2$ and $\mathbb{R}^2$ equipped with $\llangle\cdot,\cdot\rrangle$ and with $\langle\cdot,\cdot\rangle$ as the secondary metric. (The latter has the $x$-axis as an isotropic direction while in the former this role is played by the $z$-axis.) Therefore, it comes as no surprise the similarity between Theorems \ref{t2} and \ref{t3}, where distances are measured using $\langle\cdot,\cdot\rangle$, and Theorem \ref{t12}, where distances are measured using $\llangle\cdot,\cdot\rrangle$.
\end{remark}


\section{The catenary as the generating curve of minimal surfaces of revolution}\label{s-catenary2}

In this section, it  will be derived that Euclidean rotational minimal surfaces in $\mathbb{I}^3$, i.e., surfaces in $\mathbb{I}^3$ rotated around an isotropic axis, are generated by isotropic catenaries. This will extend Euler's result {\cite{BarbosaColares1986,euler}}  to the ambient of simply isotropic spaces.   Let $L_z$ be the isotropic $z$-axis. The one-parameter group $\mathcal{G}$ of rotations that leave $L_z$ point-wise fixed is $\mathcal{G}=\{\mathcal{R}_\theta:\theta\in\mathbb{R}\}$, where {\cite{daSilvaMJOU2021,Sachs1990}}
$$
\mathcal{R}_\theta = \left(
                      \begin{array}{ccc}
                      \cos\theta  & -\sin\theta & 0 \\
                      \sin\theta & \cos\theta & 0 \\
                      0 & 0 & 1
                      \end{array}
                      \right).
$$
Let $I=[a,b]\subset\mathbb{R}$, $a>0$, and let $\gamma: I\to \Pi_{xz}\subset\mathbb{R}^3$ be a smooth curve parametrized by $\gamma(t)=(t,0,z(t))$, $t>0$, where $\Pi_{xz}$ is the isotropic plane of equation $y=0$. Let $S_\gamma=\{\mathcal{R}_\theta\cdot\gamma(t):t\in I, \theta\in\mathbb{R}\}$ be the surface in $\mathbb{I}^3$ obtained by rotating $\gamma$ around the isotropic line $L_z$. 
A parametrization of $S_\gamma$ is  
\begin{equation}\label{p-sms}
    \mathbf{r}(t,\theta) =(t\cos\theta,t\sin\theta,z(t)),\quad t\in I, \theta\in\mathbb{R}.
\end{equation}
If $\mathbf{N}_{min}$ and $\mathbf{N}_{par}$ are the minimal and parabolic normal of $S_\gamma$, the relative area \eqref{def::RelAreaInI3} is given by
\begin{equation*}
    \rmd A^* = (\mathbf{N}_{par}\cdot\mathbf{N}_{min})\,\rmd A. 
\end{equation*}
Thus, the problem of minimum area for surfaces of revolution consists in   minimizing the relative area
$$\mathcal{A}[\gamma]= \int_{a}^{b}\int_{0}^{2\pi}\rmd A^* $$
among all curves $\gamma(t)=(t,0,z(t))$ for which $z(a)=r_1$ and $z(b)=r_2$ are fixed. Here, $r_1$ and $r_2$ are the radii of the circles forming the boundary of $S_\gamma$. {(The area $\int\rmd A$ must be replaced by the relative area to obtain surfaces with $H=0$ \cite{daSilvaJG2021}.)}

\begin{theorem}\label{t4} 
If $S_\gamma$ is a surface of minimum (relative) area, then  $\gamma$ is a horizontal line (and $S_\gamma$ is a horizontal plane) or $\gamma$ is an isotropic catenary with respect to $L_z$ given in Eq. \eqref{cate} for $\lambda=0$. 
\end{theorem}

\begin{proof} 
If $S_\gamma$  has minimum area, then its generating curve $\gamma$ is a critical point of the functional $\mathcal{A}[\gamma]$. Let us compute $\rmd A^*$. The first fundamental form and area element of $S_\gamma$ with respect to the parametrization $\mathbf{r}=\mathbf{r}(t,\theta)$ are given by
\begin{equation*}
    \mathrm{I} = \rmd t^2+t^2\,\rmd\theta^2 \quad \mbox{and} \quad \rmd A=t\,\rmd t\rmd\theta. 
\end{equation*}
The minimal and parabolic normal are 
\begin{equation*}
    \mathbf{N}_{min} = (-z'\cos\theta,-z'\sin\theta,1)
\end{equation*}
and
\begin{equation*}
    \mathbf{N}_{par} = \left(-z'\cos\theta,-z'\sin\theta,\frac{1}{2}-\frac{z'\,^2}{2}\right).
\end{equation*}
Therefore, the relative area is computed as
 $$
 \rmd A^* = (\mathbf{N}_{par}\cdot\mathbf{N}_{min})\,\rmd A = t\left(\frac{1}{2}+\frac{z'\,^2}{2}\right)\rmd t\,\rmd\theta. 
 $$
The relative surface area functional $\mathcal{A}$ then becomes
$$
\mathcal{A}[\gamma]= \pi \int_{a}^{b}t(1+z'^2)\,\rmd t .
$$
Notice that the Lagrangian  $J(t,z,z')=t(1+z'^2)$ does not depend on $z$. Thus,  
 the Euler-Lagrange equation \eqref{el} reduces to 
\begin{equation}\label{cate6}
    2tz' = \frac{\partial J}{\partial z'}=2c
\end{equation}
for some constant $c\in\mathbb{R}$. On the one hand, if $c=0$, then $z=z(t)$ is a constant function, which implies that $\gamma$ is a horizontal line and $S_\gamma$ is a horizontal plane. On the other hand, if $c\not=0$, an integration of Eq. \eqref{cate6} gives $z(t)=c\ln(t)+d$, where $c,d\in\mathbb{R}$.

Alternatively, we could have noticed that $J$ coincides, up to multiplication by a constant, with the Lagrangian associated with $\mathcal{F}_z$ in Theorem \ref{t1} under the condition $\lambda=0$.
\end{proof}

The Theorem \ref{t4} is analogous to Euler's result in the simply isotropic ambient space. Proceeding with the motivation provided by the catenoid, it is natural to ask whether there exists an isotropic catenoid connecting any two coaxial circles of $\mathbb{I}^3$. It is known that the existence of an Euclidean catenoid joining two coaxial circles depends on the distance between both circles \cite{bl,is,ni}. If the circles are sufficiently close (depending on the radii of the circles), then there exist two catenoids connecting both circles, but if the distance between the circle is large, then there exits no catenoid connecting them. The following result shows that this problem in $\mathbb{I}^3$ has a different solution.

\begin{theorem} Let $\Gamma_1$ and $\Gamma_2$ be two coaxial circles in $\mathbb{I}^3$ with respect to $L_z$ and with distinct radii. Then, there exists a unique isotropic catenoid with axis $L_z$  and connecting $\Gamma_1$ and $\Gamma_2$. On the other hand, if $\Gamma_1$ and $\Gamma_2$ have the same radii, then there exists no catenoid joining them.
\end{theorem}

\begin{proof} The assumption that the circles have different radii is necessary since the profile curve of the isotropic catenoid, namely the isotropic catenary, is monotonic as a function of the distance to $L_z$. Let   $(r_i,z_i)\in\mathbb{R}^2_{+}$ be the intersection of $\Gamma_i$ with the $xz$-plane,  $i=1,2$. By hypothesis,  $r_1\not=r_2$. Since the   isotropic catenary is $\gamma(t)=(t,0,c\ln(t)+d)$, the proof is completed if it is established the existence of $c,d\in\mathbb{R}$ such that 
\begin{equation}\label{cln}
\left\{
\begin{aligned}
c\ln(r_1)+d&=z_1\\
c\ln(r_2)+d&=z_2.
\end{aligned}
\right.
\end{equation}
Without loss of generality, we may assume that $z_1<z_2$ (A similar reasoning holds if $z_1>z_2$). In particular, $c>0$ in order to ensure that the function $t\mapsto c\ln(t)+d$ is increasing. 
From the first equation of \eqref{cln}, it is deduced $d=z_1-c\ln(r_1)$. By using the second equation of \eqref{cln},  the problem  reduces to finding $c>0$ such that  $c\ln(\tfrac{r_2}{r_1})+z_1=z_2$. Define the function $f(c)=c\ln(\tfrac{r_2}{r_1})+z_1$. Using that $r_2/r_1>1$,  we have
$$
\lim_{c\to 0^+}f(c)=z_1 \quad \mbox{and} \quad \lim_{c\to \infty}f(c)=\infty.
$$ 
Then, the Intermediate Value Theorem assures the existence of $c>0$ such that $f(c)=z_2$. On the other hand, since the function $f=f(c)$ is strictly increasing, the value $c$ such that $f(c)=z_2$ is unique.
\end{proof}


\begin{remark}
In isotropic space $\mathbb{I}^3$, there are two types of surfaces of revolution, namely, Euclidean rotational surfaces and parabolic rotational surfaces. The classification of minimal surfaces of revolution was done in \cite{daSilvaMJOU2021} obtaining the logarithmoid of revolution and the hyperbolic paraboloid, respectively. (See our Prop. \ref{prop::CMCparabRevSurf} for the characterization of minimal surfaces of parabolic revolution.) While the hyperbolic paraboloid is self-conjugate, the logarithmoid of revolution is conjugate to a helicoid \cite{daSilvaJG2021}. This property of the logarithmoid of revolution together with the fact that its generating curve is an isotropic catenary show that this surface is the simply isotropic analog of the catenoid.
\end{remark}

\section{Simply isotropic singular minimal surfaces}\label{s-surface}

In this section, we extend the notion of the catenary of $\mathbb{I}^2$  to the problem of the hanging surface in $\mathbb{I}^3$. Consider a surface $M^2$  of constant mass density which is suspended from a given closed curve $\Gamma$. Let $A_0>0$ be the relative area of $M^2$. The hanging surface problem  consists in finding the shape of $M^2$ when $M^2$ is suspended under its weight, and this weight is measured with respect to a plane of $\mathbb{I}^3$. As in the case of the catenary, the computation of the weight depends on whether  the reference plane is isotropic or non-isotropic. In both situations, the problem is equivalent to finding the surfaces which are critical points of the {weight}  among all surfaces with the same boundary curve $\Gamma$ and the same relative area $A_0$.

Similarly to Section \ref{intro}, for the hanging chain problem of $\mathbb{I}^2$, we shall calculate the {weight} of a surface in $\mathbb{I}^3$ using the relative area. Otherwise, the use of the area element would  lead to trivial conclusions. Indeed, if we were using the regular area element, then the weight of $M^2:\mathbf{r}(y,z)=(u(y,z),y,z)$ computed with respect to an isotropic plane $\Pi_{yz}=\{x=0\}$ would be 
$$\int_{M^2}d(\mathbf{r}(y,z),\Pi_{yz})\rmd A=\int_\Omega u\, u_z\, \rmd y\rmd z.$$
Finally, if  $u:\Omega\subset \Pi_{yz}\to\mathbb{R}$ is a smooth function over the open set $\Omega$ and $\Gamma$ is the boundary of $M^2$, then by the divergence theorem, the weight can be rewritten as
$$
\int_\Omega u\, u_z\, \rmd y\rmd z=\int_{\partial\Omega}  (0,\frac{u^2}{2})\cdot \nu\, \rmd s,
$$
where $\nu$ is the outer unit normal vector  along $\partial\Omega$. This integral depends only on $\Gamma$ and $\partial\Omega$ and, consequently,  it is constant  for all surfaces with the same boundary $\Gamma$. Similar arguments can be employed for the weight measured with respect to a non-isotropic plane. In conclusion, we must replace the area element by the relative area one.

In the following, we shall distinguish between isotropic and non-isotropic reference planes.  Let $M^2$ be an admissible surface of $\mathbb{I}^3$ parametrized in its normal form: $\mathbf{r}(x,y)=(x,y,u(x,y))$. From $
\mathbf{r}_x = (1,0,u_x)$ and $\mathbf{r}_y=(0,1,u_y)$, the first and second fundamental forms and the area element are
\begin{equation}
    \mathrm{I} = \rmd x^2+\rmd y^2, \quad \mathrm{II} = u_{xx}\rmd x^2+2u_{xy}\rmd x\rmd y+u_{yy}\rmd y^2, \quad \mbox{ and } \quad \rmd A = \rmd x\rmd y.
\end{equation}
The formula \eqref{mean} of the mean curvature $H$ of $M^2$ is 
\begin{equation}\label{eq::MeanCurvOfxyGraph}
     H = \frac{u_{xx}+u_{yy}}{2}.
\end{equation}
In addition, the minimal and parabolic normal vector fields are
\begin{equation}\label{min2}
    \mathbf{N}_{min} = \left(-u_x,-u_y,1\right) \quad \mbox{ and } \quad
    \mathbf{N}_{par} = \left(-u_x,-u_y,\frac{1}{2}-\frac{u_x^2+u_y^2}{2}\right).
\end{equation}
Hence, the relative area element of $M^2$ is
\begin{equation}\label{da2}
     \rmd A^* =  (\mathbf{N}_{min}\cdot\mathbf{N}_{par})\, \rmd A =  \frac{1}{2}\left(1+u_x^2+u_y^2\right)\, \rmd x\rmd y.  
\end{equation}

We begin considering the weight measured with respect to the isotropic plane $\Pi_{yz}$  of equation $x=0$.     The weight of $M^2$ is calculated measuring the distance between the points of $M^2$ and the plane $\Pi_{yz}$. This distance is $d((x,y,z),\Pi_{yz})=|x|$. From now on, and without loss of generality, it will be assumed that all surfaces are included in the half-space $\Pi_{yz}^+=\{(x,y,z)\in\mathbb{I}^3:x>0\}$.  The {weight} $W$ of $M^2$ is 
 \begin{equation}
    { W = \int_{M^2} x\, \rmd A^*, } 
 \end{equation}
where $\rmd A^*$ is the relative area element of $M^2$. Thus, the functional to minimize is \begin{equation}
    \mathcal{E}_{yz}[M^2]=\int_{M^2} x\, \rmd A^*-\lambda\int_{M^2}\, \rmd A^*=\int_{M^2}(x-\lambda)\, \rmd A^*,
\end{equation}
where $\lambda$ is a Lagrange multiplier due to the constraint on the relative area.   Let $\Gamma$ and $A_0$  be the boundary curve and the relative area of  $M^2$, respectively. Let $\mathcal{C}_{\Gamma,A_0}^{yz}(\Omega)$ be the class of all graphs $G$ over $\Omega$, $G\subset \Pi_{yz}^+$, with boundary $\Gamma$ and relative area $A_0$.   Let $X=\partial_x\in\mathfrak{X}(\mathbb{I}^3)$, i.e., $X$ is the velocity vector field of the lines used to measure the distance in the computation of the {weight}. 
 
 \begin{theorem} 
 Let $M^2\subset \Pi_{yz}^+$ be the graph of a function $u:\Omega\subset\Pi_{yz}\to\mathbb{R}$, where $\Omega$ is a bounded domain. Then,  $M^2$    is a critical point of $\mathcal{E}_{yz}$ in $\mathcal{C}_{\Gamma,A_0}^{yz}(\Omega)$ if, and only if, its mean curvature $H$ satisfies
 \begin{equation}\label{sms1}
    H(\mathbf{r}) =  \frac{\langle\mathbf{N}_{par},X\rangle}{2(x-\lambda)} =  \frac{\langle\mathbf{N}_{par},X\rangle}{2(d(\mathbf{r},\Pi_{yz})-\lambda)} .
\end{equation}
\end{theorem} 
    
\begin{proof}
Using Eq. \eqref{da2}, the functional $\mathcal{E}_{yz}$ is 
\[
    \mathcal{E}_{yz}[M^2] = \int_{M^2}(x-\lambda)\,\rmd A^* = \int_{\Omega}\frac{(x-\lambda)}{2}\left(1 +p^2+q^2\right)\rmd x\rmd y,
\]
where $p=u_x$ and $q=u_y$. To find the critical points of $\mathcal{E}_{yz}$ in $\mathcal{C}_{\Gamma,A_0}^{yz}$, we use the  Euler-Lagrange equation
\begin{equation}\label{el2}
 \frac{\partial F}{\partial u}-\frac{\partial}{\partial x}\left(\frac{\partial F}{\partial p}\right)-\frac{\partial}{\partial y}\left(\frac{\partial F}{\partial q}\right) = 0,
\end{equation}
where 
$$  F(x,y,u,p,q)=\frac{(x-\lambda)}{2}\left(1 +p^2+q^2\right).$$ 
Notice that $F$ does not depend on $u$. Thus, Eq. \eqref{el2}  leads to
$$ ((x-\lambda)p)_x+  ((x-\lambda)q)_y=0.$$
This gives
$$(x-\lambda)(p_x+q_y)+p=0.$$
Using Eq. \eqref{eq::MeanCurvOfxyGraph}, we have $2H(x-\lambda)=-p$. Finally, from Eq. \eqref{min2} and the definition of $X$, we conclude $\langle \mathbf{N}_{par},X\rangle=-p$, proving the validity of Eq. \eqref{sms1}.
\end{proof}
Note that Eq. \eqref{sms1} is the two-dimensional version of Eq. \eqref{cate2} by replacing the curvature $\kappa$ with the mean curvature $H$ of $M^2$. In \eqref{sms1}, the vector field $X=\partial_x$ is unitary and orthogonal to the reference plane $\Pi_{yz}$. 
 
Now, consider the weight measured with respect to the non-isotropic plane $\Pi_{xy}$ of equation $z=0$. The {weight} of $M^2$ is calculated measuring the distance between the points of $M^2$ with respect to $\Pi_{xy}$. 
Since the distance to the $xy$-plane vanishes if we consider the simply isotropic metric, we shall employ the secondary metric to compute this distance. Therefore, we have  $d((x,y,z),\Pi_{xy})=|z|$. Again, we will assume that $M^2$ is contained in one of the two half-spaces determined by $\Pi_{xy}$, which will be  $\Pi_{xy}^+=\{(x,y,z)\in\mathbb{I}^3:z>0\}$.  The {weight} of $M^2$ is 
\begin{equation}
 {    W = \int_{M^2} z\, \rmd A^*. } 
\end{equation}
The functional to minimize is
\begin{equation}
    \mathcal{E}_{xy}[M^2]=\int_{M^2} z\, \rmd A^*-\lambda\int_{M^2}\, \rmd A^*=\int_{M^2}(z-\lambda)\, \rmd A^*.
\end{equation}
Here, $\lambda$ is again a Lagrange multiplier due to the constraint on the  relative area. Let $\Gamma$ and $A_0$  be the boundary curve and the relative area of  $M^2$, respectively. Let $\mathcal{C}_{\Gamma,A_0}^{xy}(\Omega)$ be the class of all graphs $G$ over $\Omega$, $G\subset\Pi_{xy}^+$, with boundary $\Gamma$ and relative area $A_0$.   Let $Z=\partial_z\in\mathfrak{X}(\mathbb{I}^3)$, i.e., $Z$ is the velocity vector field of the lines used to measure distances in the computation of the {weight}.
 
 \begin{theorem}\label{t6} 
 Let  $M^2\subset\Pi_{xy}^+$ be the graph of a function $u:\Omega\subset\Pi_{xy}\to\mathbb{R}$, where $\Omega$ is a bounded domain. Then,  $M^2$   is  a critical point of $\mathcal{E}_{xy}$  in $\mathcal{C}_{\Gamma,A_0}^{xy}$ if, and only if, its mean curvature $H$ satisfies
 \begin{equation}\label{sms2}
    H(\mathbf{r}) =  \frac{\llangle\mathbf{N}_{par},Z\rrangle}{2(z-\lambda)} =  \frac{\llangle\mathbf{N}_{par},Z\rrangle}{2(d(\mathbf{r},\Pi_{xy})-\lambda)}.
\end{equation}
 \end{theorem} 

\begin{proof}
The functional $\mathcal{E}_{xy}$ to minimize in $\mathcal{C}_{\Gamma,A_0}^{xy}$ writes as  
\[
    \mathcal{E}_{xy}[M^2] = \int_{M^2}(z-\lambda)\,\rmd A^* = \int_{\Omega}\frac{u-\lambda}{2}\left(1+p^2+q^2\right)\rmd x\rmd y.
\]
  The corresponding Euler-Lagrange equation  is calculated with  
$$ F(x,y,u,p,q)= \frac{u-\lambda}{2}\left(1+p^2+q^2\right).$$
 Computing \eqref{el2}, we obtain 
\[
 {1+p^2+q^2} -\left(2(u-\lambda)p\right)_x-\left(2(u-\lambda)q\right)_y=0,
\]
or equivalently, 
$$1-p^2-q^2-2(u-\lambda) ( p_x+ q_y)=0.$$
This identity, together with Eq. \eqref{eq::MeanCurvOfxyGraph}, gives
$$H= \frac{1-u_x^2-u_y^2}{4(u-\lambda)}.$$
Equation \eqref{sms2} follows from the expression of $\mathbf{N}_{par}$ in Eq. \eqref{min2}: $\llangle\mathbf{N}_{par},\partial_z\rrangle=(1-p^2-q^2)/2$.
\end{proof}

Note that for the surfaces described in Theorems \ref{t6} and \ref{t5}, we may set $\lambda=0$ after a convenient translation of $\mathbb{I}^3$. We have the following definition.


\begin{definition} 
Let $\Pi\subset\mathbb{I}^3$ be a plane and $V\in\mathfrak{X}(\mathbb{I}^3)$ be a unit vector field orthogonal to $\Pi$. A surface $M^2$ in $\mathbb{I}^3$ is an isotropic singular minimal surface with respect to $\Pi$ if its mean curvature satisfies 
\[
H(p) =   \frac{\langle\mathbf{N}_{par},V\rangle}{2\, d(p,\Pi)},\quad p\in M^2,
\]
if $\Pi$ is an isotropic plane, or its mean curvature satisfies 
\[
H(p) =   \frac{\llangle\mathbf{N}_{par},V\rrangle}{2\, d(p,\Pi)},\quad p\in M^2,
\]
if $\Pi$ is a non-isotropic plane.
\end{definition}

Note that this definition extends the notion of isotropic catenary to dimension $2$. It is also possible to extend this definition if we introduce  a power $\alpha$ in the functional  $\mathcal{E}_{yz}$ and $\mathcal{E}_{xy}$, replacing $x-\lambda$ with  $x^\alpha-\lambda$  and $z-\lambda$ with $z^\alpha-\lambda$, respectively. The critical points of these functionals generalize the concept of $\alpha$-catenaries and they are characterized by  Eqs. \eqref{sms1}  and \eqref{sms2} after multiplying the right-hand sides by the factor $\alpha$. {Here, the case $\alpha=0$ corresponds to minimal surfaces.}

\section{Simply isotropic invariant singular minimal surfaces}\label{s-invariant}

In this section, we study the isotropic singular minimal surfaces that are invariant surfaces and whose generating curves lie in the $yz$-plane. We show that there exist no singular minimal helicoidal surfaces. For surfaces of revolution, either Euclidean or parabolic, we completely classify all singular minimal surfaces when the weight is measured with respect to an isotropic plane. On the other hand, when the weight is measured with respect to the a non-isotropic plane, the generating curve is associated with the solution of a non-linear ordinary differential equation of second order.

\subsection{Singular minimal surfaces of Euclidean revolution}\label{s-elliptic}

A helicoidal surface is the surface invariant under the action of the one-parameter group $\mathcal{G}_p=\{\mathcal{H}_{\theta}:\theta\in\mathbb{R}\}$ of helicoidal motions \cite{daSilvaMJOU2021}, where   
\[
\left(
 \begin{array}{c}
      x  \\
      y \\
      z \\
 \end{array}
\right)\mapsto
\mathcal{H}_{\theta} \left(
 \begin{array}{c}
      x  \\
      y \\
      z \\
 \end{array}
\right) = \left(
                \begin{array}{ccc}
                      \cos\theta   &   -\sin\theta  & 0  \\
                      \sin\theta  &   \cos\theta  & 0  \\
                      0            & 0             & 1  \\
                \end{array} 
                \right)\left(
 \begin{array}{c}
      x  \\
      y \\
      z \\
 \end{array}
\right)+\left(
 \begin{array}{c}
      0       \\
      0       \\
      c\,\theta \\
 \end{array}
\right).
\]
Applying helicoidal motions  to $\gamma(t)=(t,0,z(t))$ gives the invariant surface $S_{\gamma}=\mathcal{H}_\theta(\gamma)$ parametrized as
\begin{equation}
    \mathbf{r}(t,\theta) = (t\cos\theta,t\sin\theta,c\,\theta+z(t)).
\end{equation}
For $c=0$, $\mathcal{H}_{\theta}=\mathcal{R}_{\theta}$ and we then obtain surfaces of (Euclidean) revolution as in Eq. \eqref{p-sms}.

The mean curvature $H$ of $S_\gamma$ is  
\begin{equation}\label{h3}
H=\frac{z'+tz''}{2t}.
\end{equation}
The parabolic normal vector field of $S_\gamma$ is
\begin{equation}\label{h4}
\mathbf{N}_{par}=\left(\frac{c\sin\theta}{t}-z'\cos\theta,-\frac{c\cos\theta}{t}-z'\sin\theta,\frac12(1-\frac{c^2}{t^2}-z'^2)\right).
\end{equation}

\begin{theorem} \label{t5}
If $S_{\gamma}$ is a helicoidal singular minimal surface, then $S_{\gamma}$ must be a surface of Euclidean  revolution, i.e., $c=0$. In addition,
\begin{enumerate}
    \item If the reference plane is $\Pi_{yz}$ (isotropic), then $z(t)=\dfrac{z_2}{t}+z_1$, $z_1,z_2\in\mathbb{R}$. In particular, this includes horizontal planes.
    \item If the reference plane is $\Pi_{xz}$ (non-isotropic), then $z=z(t)>0$ satisfies
    \begin{equation}\label{e-sms}
     z''(t)+\frac{z'(t)}{t}=\frac{1-z'(t)^2}{2z(t)}. 
    \end{equation}
\end{enumerate} 
\end{theorem}
\begin{proof}
We distinguish between two cases, depending on the type of reference plane.
\begin{enumerate}
    \item The reference plane is $\Pi_{yz}$. Then, $\langle \mathbf{N}_{par},X\rangle=c\sin\theta/t-z'\cos\theta$. Since the distance to $\Pi_{yz}$ in Eq. \eqref{sms1} is the $x$-coordinate, we have  $d(p,\Pi_{yz})=t\cos\theta$ and the Eq. \eqref{sms1} for the mean curvature of a helicoidal singular minimal surface reduces to
    \[
    \frac{z'+tz''}{t}=\frac{\frac{c}{t}\sin\theta-z'\cos\theta}{t\cos\theta}, 
    \] 
    or equivalently, 
    \begin{equation}\label{ss}
    t(2z'+tz'')\cos\theta-c\sin\theta=0.
    \end{equation}
    Since  $\cos\theta$ and $\sin\theta$ are  linearly independent  functions, we deduce from \eqref{ss} that $c=0$ and, consequently, $S_\gamma$ is a surface of revolution. Moreover, $z=z(t)$ must satisfy $2z'+tz''=0$. The solution of this equation is   $z(t)=z_2t^{-1}+z_1$, where $z_1,z_2\in\mathbb{R}$.
    \item The reference plane is $\Pi_{xy}$. For the computation of Eq. \eqref{sms2}, we have $\llangle \mathbf{N}_{par},Z\rrangle=\frac{1}{2}(1-c^2/t^2-z'\,^2)$ and $d(p,\Pi_{xy})=z+c\,\theta$. Using Eq. \eqref{h3}, it follows that Eq.  \eqref{sms2} is
    $$
    \frac{z'+t z''}{2t}=\frac{\frac{1}{2}(1-\frac{c^2}{t^2}-z'^2)}{2(z+c\theta)},
    $$
    or equivalently, 
    \[
    c(z'+tz'')\theta+z(z'+tz'')-\frac{t}{2}(1-\frac{c^2}{t^2}-z'^2)=0.
    \]
    
    Since $\{1,\theta\}$ forms a set of linearly independent functions, we deduce that
    \begin{equation}\label{ss2}
    c(z'+tz'')=0 \quad \mbox{and} \quad z(z'+tz'')-\frac{t}{2}(1-\frac{c^2}{t^2}-z'^2)=0.
    \end{equation}
    From the first equation of \eqref{ss2}, we distinguish between two cases according to whether $c$ is $0$ or not. If $c\not=0$, then $z'+tz''=0$. If, in addition, $z'=0$, then the second equation of \eqref{ss2} is simply $\frac{t}{2}(1-\frac{c^2}{t^2})=0$, which leads to a contradiction, namely, $t$ cannot be constant. On the other hand, if $c\not=0$ but $z'\not=0$, then  the equation  $z'+tz''=0$ yields $z'=m/t$ for some constant $m>0$. Using this, the second equation of \eqref{ss2} is $\frac{1-m^2}{2}t-\frac{c^2}{2t}=0$, which also leads to a contradiction. Therefore, we must have $c=0$ and, in particular, $S_\gamma$ is a surface of revolution.  Coming back to \eqref{ss2}, we conclude that $z(z'+tz'')-\frac{t}{2}(1-\frac{c^2}{t^2}-z'^2)=0$, which is nothing but  Eq. \eqref{e-sms}.
\end{enumerate}
\end{proof}

The solution $\gamma$ in item 1 of Theorem \ref{t5} coincides with the $2$-catenary with respect to $L_z$ that appeared in Eq. \eqref{acate1}. This observation is a particular version  of a more general result. 

\begin{corollary} Let $S_\gamma$ be a surface of revolution in $\mathbb{I}^3$ with respect to $L_z$. Then, $S_\gamma$ is an isotropic $\alpha$-singular minimal surface with respect to $\Pi_{yz}$ if, and only if, $\gamma$ is an isotropic $(\alpha+1)$-catenary with respect to $L_z$.
\end{corollary}

\begin{proof} 
Suppose that $\gamma(t)=(t,0,z(t))$. Using the parametrization in Eq. \eqref{p-sms}, it follows from Eqs. \eqref{h3} and \eqref{h4} that Eq.  \eqref{sms1} implies
$$\frac{z'+tz''}{2t}=-\alpha\frac{z'\cos\theta}{2t\cos\theta}=-\alpha\frac{z'}{2t}.$$
Equivalently, we have $(\alpha+1)z'+tz''=0$. But this equation coincides with Eq. \eqref{cate42} for $\alpha+1$.
\end{proof}

We conclude this section by studying singular minimal surfaces of revolution with respect to non-isotropic planes and that intersect the rotation axis.  The generating curve $\gamma=\gamma(t)$ is defined over an interval $I=(t_0,t_1)\subset(0,\infty)$ and, therefore, if $t_0=0$, then $\gamma$ meets the rotation axis $L_z$. 
We are interested in the configuration where the intersection is orthogonal, which implies that the surface $S_\gamma$ that generates $\gamma$ is smooth at $S_\gamma\cap L_z$. 

The height function $z=z(t)$ of $\gamma(t)=(t,0,z(t))$  satisfies Eq. \eqref{e-sms}. For $a,b\in\mathbb{R}$, $a>0$, consider the following Initial Value Problem (IVP)
\begin{equation}\label{e2-sms}
\left\{
\begin{aligned}
&z''(t)+\frac{z'(t)}{t}=\frac{1-z'(t)^2}{2z(t)},\\
&z(t_0)=a,\ z'(t_0)=b.
\end{aligned}\right.
    \end{equation}
For any $t_0>0$, the existence and uniqueness of the IVP \eqref{e2-sms} is assured by the standard theory of Ordinary Differential Equations. However, if $t_0 =0$  the equation \eqref{e2-sms} is degenerate at $t=0$ and the existence of a solution may be lost. To establish the existence of a surface of revolution intersecting orthogonally the axis $L_z$, it is necessary to solve the IVP \eqref{e2-sms} at $t=0$ under the condition $z'(0)=0$.  

\begin{theorem} For any $a>0$, the initial value problem \eqref{e2-sms} with initial conditions $z(0)=a$ and $z'(0)=0$ has a solution $z\in C^2([0,R])$ for some $R>0$. In addition, the solution depends continuously on the parameter $a$.
\end{theorem}
\begin{proof} Multiplying Eq. \eqref{e-sms} by  $t$, the equation becomes $(tz')'=t(1-z'^2)/(2z)$. Define the operator
$$
(\mathsf{T}z)(t)=a+\int_0^t \frac{1}{r}\left(\int_0^r\frac{\tau(1-z'(\tau)^2)}{2z(\tau)}\, \rmd \tau\right)\, \rmd r, \quad z(t) \in C^1([0,R]).
$$
It is immediate that   $z$ is a solution of the problem \eqref{e2-sms} with $z'(0)=0$ if $u$ is a fixed point of the operator $\mathsf{T}$. Let  $C^1([0,R])$ be considered as a Banach space  endowed with the usual norm $\|z\|=\|z\|_\infty+\|z'\|_\infty$. It will be proved  the existence of $R>0$ such that $\mathsf{T}$ is a contraction in  some closed ball $\overline{\mathcal{B}(a,\epsilon)}$.    First, we prove that $\mathsf{T}$ is a self-map in a closed ball $\overline{\mathcal{B}(a,\epsilon)}$ for some $\epsilon>0$ and next, that $\mathsf{T}$ is a contraction.
\begin{enumerate}
\item Claim: there exists $\epsilon>0$ such that $\mathsf{T}(\overline{\mathcal{B}(a,\epsilon)})\subset \overline{\mathcal{B}(a,\epsilon)}$. {Indeed,} let $\epsilon>0$ such that $\epsilon<a$, which is fixed. Consider $R>0$ such that 
$$ R\leq \min\left\{\sqrt{ \frac{4\epsilon(a-\epsilon)}{1+\epsilon^2}}, \frac{2\epsilon(a-\epsilon)}{1+\epsilon^2}\right\}.$$
If $z\in \overline{\mathcal{B}(a,\epsilon)}$, we have $|z(t)-a|\leq\epsilon$ and $|z'(t)|\leq\epsilon$ for $t\in [0,R]$. Then
\begin{equation*}
\begin{aligned}
|(\mathsf{T}z)(t)-a|&\leq \int_0^t \frac{1}{r}\left\vert\int_0^r\frac{\tau(1-z'(\tau)^2)}{2z(\tau)}\, \rmd \tau\right\vert\, \rmd r\leq \int_0^t\frac{1}{r}\int_0^r\frac{\tau(1+\epsilon^2)}{2(a-\epsilon)}\, \rmd \tau \,\rmd r\\
&=\frac{t^2(1+\epsilon^2)}{8(a-\epsilon)}\leq \frac{R^2(1+\epsilon^2)}{8(a-\epsilon)}\leq\frac{\epsilon}{2}. 
\end{aligned}
\end{equation*}
On the other hand, 
\begin{equation*}
\begin{aligned}
|(\mathsf{T}z-a)'(t)|&\leq  \frac{1}{t}\int_0^t\frac{\tau(1+\epsilon^2)}{2(a-\epsilon)}\, \rmd \tau=\frac{t(1+\epsilon^2)}{4(a-\epsilon)}\leq \frac{R(1+\epsilon^2)}{4(a-\epsilon)}\leq\frac{\epsilon}{2}. 
\end{aligned}
\end{equation*}
This proves that $\| \mathsf{T}z-a\|\leq\epsilon$. Hence, $\mathsf{T}z\in\overline{\mathcal{B}(a,\epsilon)}$. As a consequence of the claim, the operator $\mathsf{T}$ is a self-map $\mathsf{T}\,\colon\, \overline{\mathcal{B}(a,\epsilon)}\to \overline{\mathcal{B}(a,\epsilon)} $.

\item Claim: the operator $\mathsf{T}:\overline{\mathcal{B}(a,\epsilon)}\to \overline{\mathcal{B}(a,\epsilon)}$ is a contraction. {Indeed,} let $L_1$  and $L_2$ respectively denote the Lipschitz constants of the functions $x\mapsto 1/(2x)$ and $x\mapsto 1-x^2$ in $[a-\epsilon,a+\epsilon]$, provided that $\epsilon<a$. Let $L=L_1L_2$. For all $z_1,z_2\in C^1([0,R])$, we have  
$$
\|\mathsf{T}z_1-\mathsf{T}z_2\|=\|\mathsf{T}z_1-\mathsf{T}z_2\|_\infty+\|(\mathsf{T}z_1)'-(\mathsf{T}z_2)'\|_\infty.
$$
We  study the term   $\|\mathsf{T}z_1-\mathsf{T}z_2\|_\infty$. Let $z_1,z_2$ be two functions in the ball $\overline{\mathcal{B}(a,\epsilon)}$ of $ (C^1([0,R]),\|\cdot\|)$. For all $t\in [0,R]$, where $R$ will be determined later, we have  
\begin{equation}\label{e1}
|(\mathsf{T}z_1)(t)-(\mathsf{T}z_2)(t)|\leq \int_0^t\frac{1}{r}\left(\int_0^rL\|z_1-z_2\|_\infty \tau\, \rmd \tau\right)\rmd r= \frac{Lt^2}{4}\|z_1-z_2\|_\infty\leq \frac{Lt^2}{4}\|z_1-z_2\|.
\end{equation}
Similarly, for $\|(\mathsf{T}z_1)'-(\mathsf{T}z_2)'\|_\infty$, we have
\begin{equation}\label{e2}
|(\mathsf{T}z_1)'(t)-(\mathsf{T}z_2)'(t)|\leq \frac{1}{t}\int_0^t L\|z_1-z_2\|_\infty \tau\, \rmd \tau= \frac{Lt}{2}\|z_1-z_2\|_\infty\leq \frac{Lt}{2}\|z_1-z_2\|. 
\end{equation}
Let us choose $R$   such that $R\leq\{\sqrt{2/L},1/L\}$. Then, the inequalities  \eqref{e1} and \eqref{e2} imply $\|\mathsf{T}z_1-\mathsf{T}z_2\|< \|z_1- z_2\|$, 
proving that   the operator $\mathsf{T}$ is a contraction in   $C^1([0,R])$. 
\end{enumerate}
Once the two claims have been proved, the Fixed Point Theorem asserts  the existence of  a fixed point $z\in C^1([0,R])\cap C^2((0,R])$.  This function $z$ is then a solution of \eqref{e2-sms} with $z'(0)=0$. Finally, we prove that the solution $u$ extends with $C^2$-regularity at $t=0$. By taking limits in \eqref{e2-sms} as $t\to 0$,  and by L'H\^{o}pital rule  on the quotient $z'(t)/t$, we conclude  
$$\frac{1}{2a}=\lim_{t\to 0}\frac{1-z'(t)^2}{2z(t)}=\lim_{t\to 0}z''(0)+\lim_{t\to 0}\frac{z'(t)}{t}=\lim_{t\to 0}z''(0)+\lim_{t\to 0}z''(t)=2\lim_{t\to 0}z''(0).$$
This proves that $z''(0)=1/(4a)$. The continuous dependence of local solutions on $a$ a is a consequence of the continuous dependence of the fixed points of $\mathsf{T}$ on the parameter $a$.
\end{proof}

\subsection{Singular minimal surfaces of parabolic revolution}\label{s-parabolic}

A surface of parabolic revolution is invariant under the action of the one-parameter group $\mathcal{G}_p=\{\mathcal{P}_{\theta}:\theta\in\mathbb{R}\}$ of parabolic revolutions \cite{daSilvaMJOU2021}, where  

\[
\left(
 \begin{array}{c}
      x  \\
      y \\
      z \\
 \end{array}
\right)\mapsto
\mathcal{P}_{\theta} \left(
 \begin{array}{c}
      x  \\
      y \\
      z \\
 \end{array}
\right) = \left(
                \begin{array}{ccc}
                      1  &   0  & 0  \\
                      0  &   1  & 0  \\
                    c_1\theta & c_2\theta & 1  \\
                \end{array} 
                \right)\left(
 \begin{array}{c}
      x  \\
      y \\
      z \\
 \end{array}
\right)+\left(
 \begin{array}{c}
      a\theta  \\
      b\theta \\
      c\,\theta+\frac{ac_1+bc_2}{2}\theta^2 \\
 \end{array}
\right).
\]
Applying parabolic revolutions to $\gamma(t)=(t,0,z(t))$ gives the invariant surface $S_{\gamma}=\mathcal{P}_\theta(\gamma)$ parametrized as
\begin{equation}\label{eq::ParametrizParabRevSurf}
    \mathbf{r}(t,\theta) = \left(a\theta+t,b\theta,c\theta+\frac{ac_1+bc_2}{2}\theta^2+c_1t\theta+z(t)\right).
\end{equation}
The parabolic surface of revolution $S_{\gamma}$ is regular provided that $b\not=0$. If $ac_1+bc_2=0$, then $S_{\gamma}$ is called a surface of warped translation (see Figure \ref{fig:SingMinSurf}).

The mean curvature $H$ of $S_\gamma$ is  
\begin{equation}\label{HParabRevSurf}
H= \frac{a^2+b^2}{2b^2}z''+\frac{bc_2-ac_1}{2b^2}.
\end{equation}
The parabolic normal vector field of $S_\gamma$ is
\begin{equation}\label{MparParabRevSurf}
\mathbf{N}_{par} = \left(-c_1\theta-z',\frac{az'-bc_2\theta-c-c_1t}{b}, \frac{1}{2}-\frac{F(t)}{2}\right),
\end{equation}
where
\begin{equation}\label{ff}
    F(t) = \frac{(c+c_1t)^2}{b^2}-\frac{2a(c+c_1t)}{b^2}z'+\frac{a^2+b^2}{b^2}z'^2-\frac{2t}{b}[(ac_2-bc_1)z'-c_2(c+c_1t)]+t^2(c_1^2+c_2^2).
\end{equation}

\begin{figure}
    \centering
    \includegraphics[width=1.0\linewidth]{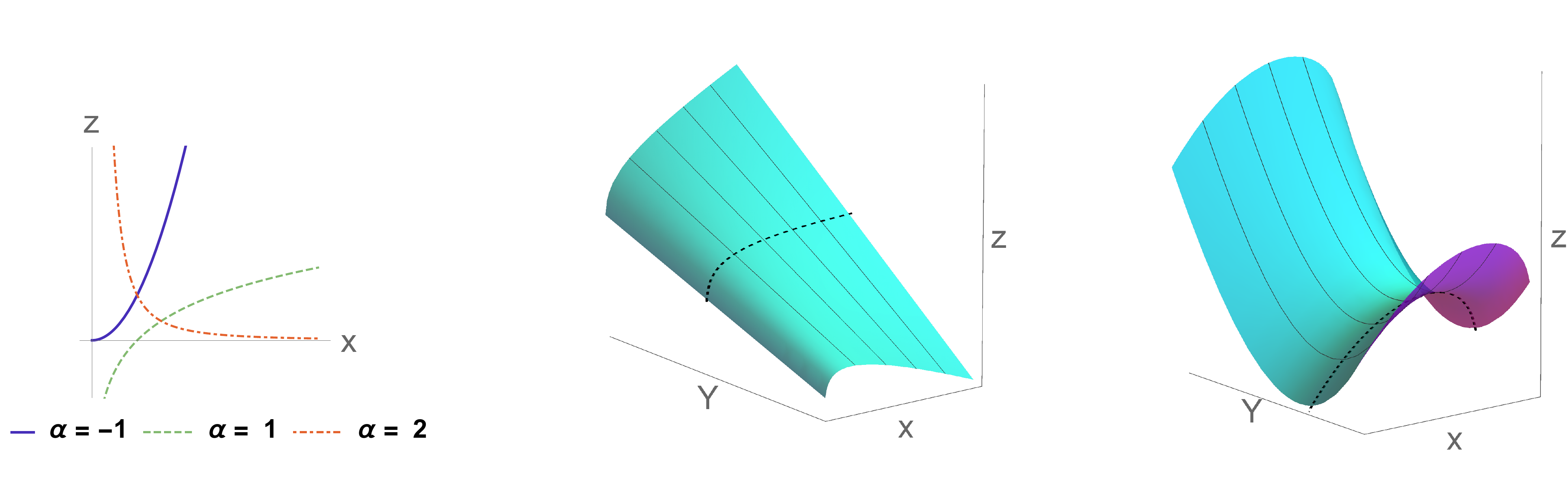}
    \caption{$\alpha$-catenaries and singular minimal surfaces. (Left) The $\alpha$-catenaries, Eq. \eqref{acate1}, with $\alpha=-1$, $\alpha=1$, and $\alpha=2$. (Center)  Singular minimal surface obtained as a warped translation surface, Eq. \eqref{eq::ParametrizParabRevSurf}, and whose generating curve (dashed black line) is a catenary $z(t)=z_2\ln t+z_1$ (See Theorem \ref{thr::SingMinSurfParabRev}). Note the surface is ruled. (Right) Singular minimal surface obtained as a surface of parabolic revolution, Eq. \eqref{eq::ParametrizParabRevSurf}, and whose generating curve (dashed black line) is a combination of a $(-1)$-catenary and a catenary: $z(t)=-\frac{c_2}{4b}t^2+z_2\ln t+z_1$ (See Theorem \ref{thr::SingMinSurfParabRev}). Here, the surface is foliated by isotropic circles. (In the figures, the parameters take the values: (Center) $c=1$ and $d=0$; (Center) $a=0,b=1,c=0,c_1=0$, and $c_2=0$; and (Right) $a=0,b=1,c=\frac{1}{2},c_1=0$, and $c_2=1.0$.) Figures generated with Mathematica.}
    \label{fig:SingMinSurf}
\end{figure}

\begin{theorem} \label{thr::SingMinSurfParabRev}
Let $S_{\gamma}$ be a singular minimal surface of parabolic revolution. (See Figure \ref{fig:SingMinSurf}.)
\begin{enumerate}
    \item If the reference plane is $\Pi_{yz}$ (isotropic), then we have two cases:
    \begin{enumerate}
        \item if $a=0$, then {$c_1=0$ and}
    \begin{equation}
         z(t) = -\frac{c_2}{4b}\,t^2+z_2\ln(t)+z_1,\quad z_1,z_2\in\mathbb{R}
    \end{equation}
\item if $a\not=0$, then $ac_2+2bc_1 = 0$ and 
    \begin{equation}
         z(t) = \frac{c_1}{2a}\,t^2+z_1, \quad z_1\in\mathbb{R}.
    \end{equation}
    \end{enumerate}
    \item If the reference plane is $\Pi_{xz}$ (non-isotropic), then   $c=c_1=0$ and  $z(t)$ is a solution of  
    \begin{equation}\label{zz}
        (2z+bc_2t^2)z''+z'^2-\frac{2abc_2}{a^2+b^2}\,tz'+\frac{2bc_2}{a^2+b^2}(z+bc_2t^2)-\frac{b^2}{a^2+b^2}=0.
    \end{equation}
\end{enumerate} 
\end{theorem}
\begin{proof}
We distinguish between two cases, depending on the type of the reference plane. In what follows, we employ the shorthand notation $A = \frac{a^2+b^2}{b^2}$ and $B = \frac{bc_2-ac_1}{b^2}$, which implies $H=(Az''+B)/2$.
\begin{enumerate}
    \item The reference plane is $\Pi_{yz}$. Then, $\langle \mathbf{N}_{par},X\rangle=-c_1\theta-z'$. Since the distance to $\Pi_{yz}$ in Eq. \eqref{sms1} is the $x$-coordinate, we have  $d(p,\Pi_{yz})=t+a\theta$. Equation \eqref{sms1} for the mean curvature of a singular minimal surface of parabolic revolution reduces to
    \[
    \frac{Az''+B}{2}=-\frac{c_1\theta+z'}{2(t+a\theta)},  
    \] 
   or equivalently,  
   \[(aAz''+aB+c_1)\theta+(Atz''+z'+tB)=0.\]
   Since  $\{1,\theta\}$ forms a set of linearly independent functions, we find 
    \begin{equation}\label{systemSingMinSurfParabRev}
        \left\{
        \begin{array}{l}
             az''+\dfrac{aB+c_1}{A} = 0  \\ [5pt]
             z''+\dfrac{z'}{tA}+\dfrac{B}{A} = 0 
        \end{array}
        \right..
    \end{equation}
    We have two sub-cases.
    \begin{enumerate}
        \item Case   $a=0$. In particular,   $A=1$ and $B=c_2/b$. Moreover,  $c_1=0$ and  the second equation of \eqref{systemSingMinSurfParabRev} is $ 
     z''+\dfrac{z'}{t}+\dfrac{c_2}{b} = 0$, whose solution is 
     \[ z(t) = -\dfrac{c_2}{4b}\,t^2+z_2\ln(t)+z_1,\quad z_1,z_2\in\mathbb{R}.\]
    \item Case  $a\not=0$. The first equation of the system \eqref{systemSingMinSurfParabRev} gives
    \[
      z(t) = -\frac{aB+c_1}{2aA}\,t^2+z_2t+z_1,\quad z_1,z_2\in\mathbb{R}.
    \]
    Substituting this solution for $z(t)$ in the second equation of \eqref{systemSingMinSurfParabRev}, we find
    \[
      -\frac{aB+c_1}{aA} + \frac{z_1-\frac{aB+c_1}{aA}t}{tA}+\frac{B}{A} = 0,\]
      or equivalently, 
      \[ \frac{z_1}{t}-\frac{aB+c_1}{a}-\frac{aB+c_1}{aA}+B = 0.
    \]
    Thus, we must have $z_2=0$ and $0=aAB-A(aB+c_1)-aB-c_1=-(c_1+Ac_1+aB)$. From the definition of $A$ and $B$, we conclude that the parameters characterizing a singular minimal surface of parabolic revolution are subjected to the constraint $ac_2+2bc_1=0$. Finally, using that $ac_2+bc_1=-bc_1$ implies that the solution for $z(t)$ takes the form
    \[
     z(t)=-\frac{aB+c_1}{2aA}\,t^2+z_1=\frac{c_1}{2a}\,t^2+z_1,\quad z_1\in\mathbb{R}.
    \]
    \end{enumerate}
    \item The reference plane is $\Pi_{xy}$. For the computation of Eq. \eqref{sms2}, we have $\llangle \mathbf{N}_{par},Z\rrangle=\frac{1}{2}(1-F(t))$ and $d(p,\Pi_{xy})=c\,\theta+\frac{ac_1+bc_2}{2}t^2+c_1t\theta+z$. Using Eq. \eqref{HParabRevSurf}, it follows that Eq.  \eqref{sms2} is
    $$
    B+A z''=\frac{1-F(t;z,z')}{2\left[(c+c_1t)\theta+\frac{1}{2}(ac_1+bc_2)t^2+z\right]}, 
    $$
    or equivalently, 
    \[  
    2\theta(c+c_1t)(Az''+B)=1-F-\left[2z+(ac_1+bc_2)t^2\right](Az''+B).\]
   Using again that $\{1,\theta\}$ is linearly independent,  we have 
   \begin{equation}\label{ss3}
   2(c+c_1t)(Az''+B) =0\quad \mbox{and}\quad F-1+[2z+(ac_1+bc_2)t^2](Az''+B)=0.
   \end{equation}
   \begin{enumerate}
       \item Case $Az''+B\not=0$ at some point $t=t_0$. Around an interval of $t_0$, we have  $c=c_1=0$. In particular, from \eqref{ff}, the function $F$ is  $F=Az'^2-\frac{2atc_2}{b}z'+t^2c_2^2$. With this value of $F$,   the second equation of \eqref{ss3} is   
    \[
        (2z+bc_2t^2)z''+z'^2-\frac{2abc_2}{a^2+b^2}\,tz'+\frac{b^2c_2^2}{a^2+b^2}t^2+\frac{bc_2}{a^2+b^2}(2z+bc_2t^2)-\frac{b^2}{a^2+b^2}=0.
    \]
    This equation is just \eqref{zz}.
    \item Case  $Az''+B=0$. Then 
    \begin{equation*}\label{ss4}
        z(t) = -\frac{B}{2A}\,t^2+z_2t+z_1=\frac{ac_1-bc_2}{a^2+b^2}\,\frac{t^2}{2}+z_2t+z_1,\quad z_1,z_2\in\mathbb{R}.
        \end{equation*}
    The second equation of \eqref{ss3} gives $F-1=0$. From the value of $F$ in \eqref{ff}, we find a polynomial of degree 2 in $t$, namely,  $f_2t^2+f_1t+f_0-1=0$. This implies   $f_2=0$, $f_1=1$, and $f_0=1$. The quadratic coefficient $f_2$ simplifies to
    \[
    \frac{(c_1^2+c_2^2)[b^2+(1+a)^2]}{a^2+b^2} = 0.
    \]
    However, this constraint would imply $b=0$, which is not allowed by hypothesis. This proves that the case $Az''+B=0$ is not possible.
    \end{enumerate}
\end{enumerate}
\end{proof}

If we consider an isotropic reference plane, the family of generating curves of singular minimal surfaces of parabolic revolution contains simply isotropic catenaries when $c_2=0$ and also isotropic circles when $z_2=0$. In the former class, we have surfaces of warped translation (see Figure \ref{fig:SingMinSurf}, Center). In the latter, the surfaces have constant mean curvature: $H=\frac{c_2}{2b}$ if $a=0$ and $H=-\frac{c_1}{2a}$ if otherwise. Constant mean curvature surfaces of parabolic revolution are implicitly defined as parabolic quadrics. From the proposition to be proved below, the type of the parabolic quadric is determined by a parameter $\Lambda$, which is $\Lambda=-c_1^2$ if $a=0$ and $\Lambda=-2c_1^2-\frac{c_2^2}{2}$ if otherwise. In both cases, the corresponding surface of parabolic revolution is a hyperbolic paraboloid.

\begin{proposition}\label{prop::CMCparabRevSurf}
An admissible surface $M^2\subset\mathbb{I}^3$ has constant mean curvature $H_0$ if, and only if, it is a parabolic quadric, where the parabolas that foliate the surfaces are isotropic circles, i.e., parabolas whose axes are an isotropic line. In addition, the type of the parabolic quadric is determined by the parameter $\Lambda=2(ac_1+bc_2)H_0-(c_1^2+c_2^2)$: $M^2$ is an elliptic paraboloid if $\Lambda>0$; $M^2$ is a parabolic cylinder if $\Lambda=0$; and $M^2$ is a hyperbolic paraboloid if $\Lambda<0$. Finally, the only minimal surfaces of parabolic revolution in $\mathbb{I}^3$ are the hyperbolic paraboloids implicitly defined by $z=\lambda(x^2-y^2)$.  
\end{proposition}
\begin{proof}
Let $M^2$ be a surface of parabolic revolution with constant mean curvature (CMC) $H_0$. Then, $M^2$ is generated by an isotropic circle $z(t)=z_0+z_1t+z_2t^2$, where the mean curvature $H_0$ depends on the parameters defining the surface by the relation $z_2=(ac_1-bc_2+2b^2H_0)/2(a^2+b^2)$ (see Ref. \cite{daSilvaMJOU2021}, Example 5.4). These surfaces are implicitly given by the equation
\begin{equation*}
    (x,y,z)\in\mathbb{I}^3:z-z_0=Ax^2+2Bxy+Cy^2+Dx+Ey,
\end{equation*}
where the coefficients are (see Ref. \cite{KelleciJMAA2021}, Fig. 3)
$$A=z_2,\quad B=\frac{c_1-2az_2}{2b},\quad C= \frac{2a^2z_2-ac_1+bc_2}{2b^2},\quad D=z_1,\quad E=\frac{c-az_1}{b}.$$
The type of the CMC surface is determined by the sign of the parameter $\Lambda=2(ac_1+bc_2)H_0-(c_1^2+c_2^2)$: elliptic paraboloid if $\Lambda>0$; parabolic cylinder if $\Lambda=0$; and hyperbolic paraboloid if $\Lambda<0$. For minimal surfaces, $\Lambda<0$ and, therefore, the surface is a hyperbolic paraboloid. Finally, the trace of the quadratic part defining the surface vanishes, which implies that a minimal surface of parabolic revolution is implicitly given, up to rigid motions, by the equation $z=\lambda (x^2-y^2)$.
\end{proof}


\begin{funding}
  Luiz da Silva acknowledges the support provided by the Mor\'a Miriam Rozen Gerber fellowship for Brazilian postdocs and the Faculty of Physics Postdoctoral Excellence Fellowship.
  Rafael  L\'opez  is a member of the Institute of Mathematics  of the University of Granada. This work has been partially supported by  the Projects  I+D+i PID2020-117868GB-I00, supported by MCIN/ AEI/10.13039/501100011033/,  A-FQM-139-UGR18 and P18-FR-4049.
\end{funding}

\begin{thebibliography}{9}











\bibitem{BarbosaColares1986}
{J. L. M. Barbosa and A. G. Colares, \emph{Minimal Surfaces in $\mathbb{R}^3$}, Springer, Berlin Heidelberg, 1986.}

\bibitem{bl} 
G. Bliss, \emph{Lectures on the Calculus of Variations}, University of Chicago Press, Chicago, 1946.
 
\bibitem{si1} L. C. B. da~Silva, The geometry of Gauss map and shape operator in simply isotropic and pseudo-isotropic spaces, \emph{J. Geom.} \textbf{110} (2019), 31.

\bibitem{daSilvaMJOU2021}
L. C. B. da~Silva, Differential geometry of invariant surfaces in simply isotropic and pseudo-isotropic spaces,
\emph{Math. J. Okayama Univ.} \textbf{63} (2021), 15.

\bibitem{daSilvaJG2021}
L. C. B. da~Silva, Holomorphic representation of minimal surfaces in simply isotropic space. \emph{J. Geom.} \textbf{112} (2021), 35.

\bibitem{euler} 
L. Euler, Methodus inveniendi curvas lineas maximi minimive proprietate gaudentes sive solution problematis isoperimetrici latissimo sensu accepti, Lausanne. Reprinted as Opera omnia Ser. 1, V. 24 (1952).

\bibitem{is} 
C. Isenberg, \emph{The Science of Soap Films and Soap Bubbles}, Dover Publications, Inc., 1992.
 
\bibitem{KelleciJMAA2021}
A. Kelleci and L. C. B. da~Silva, Invariant surfaces with coordinate finite-type Gauss map in simply isotropic space, \emph{J. Math. Anal. Appl.} \textbf{495} (2021), 124673.

\bibitem{Dierkes1990}
U. Dierkes, Singular Minimal Surfaces, in: S. Hildebrandt and H. Karcher (eds), \emph{Geometric Analysis and Nonlinear Partial Differential Equations}. Springer, Berlin, Heidelberg (2003), 177--193. 

\bibitem{MullerMonatsh}
E. M\"{u}ller, Relative Minimalfl\"{a}chen. \emph{Monatshefte f\"{u}r Mathematik und Physik} \textbf{31} (1921), 3--19. 

\bibitem{LopezAGAG2018}
R. L\'opez, Invariant singular minimal surfaces, \emph{Ann. Glob. Anal. Geom.} \textbf{53} (2018), 521--541.

\bibitem{Lopez2022CatenarySpaceForms}
R. L\'opez, The catenary in space forms, \emph{e-print} arXiv:2208.13694.

\bibitem{Lopez2022CatenaryLorentzSpaceForms}
R. L\'opez, A characterization of rotational minimal surfaces in the de Sitter space, \emph{e-print} arXiv:2208.13698.

\bibitem{ni} 
J. C. C. Nitsche, \emph{Lectures on Minimal Surfaces}, Cambridge University Press, Cambridge, 1989.

\bibitem{Sachs1987}
H. Sachs, \emph{Ebene Isotrope Geometrie}, Friedr. Vieweg \& Sohn, Brauschweig, 1987.

\bibitem{Sachs1990}
H. Sachs, \emph{Isotrope Geometrie des Raumes}, Friedr. Vieweg \& Sohn, Brauschweig, 1990.

\bibitem{Simon1991}
U. Simon, A. Schwenk-Schellschmidt, and H. Viesel, \emph{Introduction to the Affine Differential Geometry of Hypersurfaces}, Science University of Tokyo, Tokyo, 1991.

\bibitem{StruveRM2005}
R. Struve, Orthogonal Cayley-Klein groups, \textit{Results. Math.} \textbf{48} (2005), 168.

\bibitem{VerpoortAG2012}
S. Verpoort, A characterisation of Manhart's relative normal vector fields, \textit{Adv. Geom.} \textbf{12} (2012), 29--42.

\end{thebibliography}

\end{document}